\documentclass
{amsart}
\usepackage{amsmath,latexsym,amsthm,amsfonts,tipa}
\usepackage{amssymb,amscd,graphpap}
\usepackage{enumerate}
\RequirePackage{hyperref}

\usepackage[a4paper,top=1.45cm,bottom=1.55cm,left=1.59cm,right=1.59cm]{geometry}

\input xy
\xyoption{all}
\theoremstyle{plain}

\theoremstyle
{plain}
\newtheorem{theorem}{Theorem}[section]

\newtheorem{proposition}[theorem]{Proposition}

\newtheorem{fact}[theorem]{Fact}
\newtheorem{lemma}[theorem]{Lemma}
\newtheorem{corollary}[theorem]{Corollary}

\newtheorem{question}[theorem]{Question}
\newtheorem{claim}[theorem]{Claim}
\theoremstyle{definition}
\newtheorem{definition}[theorem]{Definition}

\newtheorem{example}[theorem]{Example}
\newtheorem{remark}[theorem]{Remark}

\newtheorem{problem}[theorem]{Problem}
\newtheorem{theoremintro}{Theorem}

\makeindex

\newcommand{\N}{\mathbb{N}}
\newcommand{\Z}{\mathbb{Z}}
\newcommand{\Q}{\mathbb{Q}}
\newcommand{\R}{\mathbb{R}}

\newcommand{\T}{\mathbb{T}}

\newcommand{\BB}{\mathcal B}
\newcommand{\BBB}{\mathcal B}

\newcommand{\II}{\mathcal I}

\newcommand{\QQ}{\mathcal Q}

\DeclareMathOperator{\asdim}{asdim}

\DeclareMathOperator{\dsc}{dsc}

\DeclareMathOperator{\sat}{sat}
\DeclareMathOperator{\supt}{supt}

\def\iso{{\mathcal I}soL}
\def\aaa{\approx}
\def\LL{\ell\mbox{-}\mathcal L}

\author{D. Dikranjan, I. Protasov, N. Zava}
\title{Hyperballeans of groups
}
\thanks{The first named author was partially supported by grant PSD-2015-2017-DIMA-PRID-DIKRANJAN TokaDyMA of Udine University.}
\date{}
\address{Department of mathematical, Computer and Physical Sciences, Udine University, 33 100 Udine, Italy}
\email{dikran.dikranjan@uniud.it}
\address{Department of Computer Science and Cybernetics, Kyiv University, Volodymyrska 64, 01033, Kyiv, Ukraine}
\email{i.v.protasov@gmail.com; }
\address{Department of mathematical, Computer and Physical Sciences, Udine University, 33 100 Udine, Italy}
\email{nicolo.zava@gmail.com}

\begin{document}
\maketitle

\begin{abstract}
In this paper we define some ballean structure on the power set of a group and, in particular, we study the subballean with support the lattice of all its subgroups. If $G$ is a group, we denote by $L(G)$ the family of all subgroups of $G$. For two groups $G$ and $H$, we relate
 their algebraic structure via the ballean structure of $L(G)$ and $L(H)$.
 \vspace{2 mm}
	
{\bf MSC} :54E15, 20F65, 20F15, 20E99
\vspace{2mm}
	
{\bf Keywords} : ballean, coarse structure, coarse map, asymorphism, asymptotic dimension, balleas defined by ideals, hyperballeans.
\end{abstract}


\section*{Introduction}

Coarse geometry is the study of large-scale properties of spaces, ignoring their local, small-scale, ones. It was initially developed for metric spaces and it found important applications to Novikov conjecture, to coarse Baum-Connes conjecture, and to geometric group theory, after Gromov breakthrough. Inspired by uniformities, Roe defined coarse spaces in order to encode many large-scale properties of metric spaces (\cite{Roe}). At the same time, Protasov and Banakh (\cite{ProBan}) defined balleans, an equivalent construction that generalises balls in a metric space. For a categorical look at the balleans and coarse spaces 
see \cite{DikZa}. On group $G$, one can define some ballean structures that agree with the algebraic structure of the group. Of particular
interest is the finitary ballean $\BB_G$, generated by the family of all finite subsets of a group $G$.

A relevant issue is the exploration of properties in coarse geometry related to well-known ones in topology (see, for example, \cite{CeDyVa}). For instance, there is some
evidences that connectedness in the framework of balleans is the large-scale counterpart of the Hausdorff property in topological and uniform spaces. Another
outstanding example of this approach is the asymptotic dimension defined by Gromov, inspired by the classical covering (Lebesgue) dimension.

 We now mention a further instance of this approach. For every given ballean $\BB$, in \cite{ProPro} the authors introduced the hyperballean $\BB^\flat$ of $\BB$, which is a ballean structure on the family of all non-empty bounded subsets of $\BB$. In \cite{b5}, the ballean structure $\BB^\flat$ was extended to the whole power set of $\BB$ by introducing the ballean structure $\exp\BB$. The definition of $\exp\BB$ was inspired by the theory of uniform spaces. In fact, if $X$ is a uniform space, the Hausdorff-Bourbaki hyperspace is a uniform structure on the power set of $X$ and it extends the uniform hyperspace of $X$, which is the restriction of the Hausdorff-Bourbaki hyperspace to the family of all closed subsets of $X$ (see, for example, \cite{Isb} for a discussion about uniform hyperspaces, and \cite{b5} for further details about similarities between uniform hyperspaces and hyperballeans).

Recall that the Hausdorff-Bourbaki hyperspace of a Hausdorff uniform space is not Hausdorff in general. Similarly, as we may have expected, in general, $\exp\BB$ is highly disconnected even if $\BB$ is connected (see \cite{b5} and Proposition \ref{prop:dsc_expG}). In order to obtain a more manageable object there are two approaches. The first one is focusing the attention on $\BB^\flat$. In fact, $\BB^\flat$ is connected whenever $\BB$ is connected. Alternatively, if we start from a ballean $\BB_G$ defined on a group $G$, we can consider the subballean $\mathcal L(G)=\exp\BB_G|_{L(G)}$, where $L(G)$ is the family of all subgroups of $G$. This second idea is developed in this paper.

Here, we mainly focus on two ballean structures on the subgroup lattice $L(G)$ of a group $G$. The first one, denoted by $\mathcal L(G)$,  is called the {\em subgroup exponential hyperballean}, while the second one, denoted by $\LL(G)$,  is called the {\em subgroup logarithmic hyperballean}. The latter can be characterised as follows: it is the ballean structure on $L(G)$ induced by the extended-metric
$$
d(H,K)=\log(\max\{\lvert H:H\cap\ K\rvert,\lvert K:H\cap K\rvert\}),
$$
where $H$ and $K$ are two subgroups of $G$. Actually, we provide a ballean structure, $\ell\mbox{-}\exp\BB_G$, on the entire power set of $G$ such that $\LL(G)$ is
the restriction of $\ell\mbox{-}\exp\BB_G$ to ${L(G)}$. The balleans $\mathcal L(G)$ and $\LL(G)$ have also the same connected components determined by the property that two subgroups are in the same connected component if and only if they are commensurable. In particular, the existence of isolated points (i.e., points in $L(G)$ whose connected components is just a singleton) is closely related to divisibility. In fact, we show that a subgroup $H$ of $G$ is isolated if and only if it is divisible and has a torsion-free direct summand.

Moreover, while all examples of subgroup exponential hyperballeans we considered have asymptotic dimension $0$, the behaviour of asymptotic dimension of subgroup logarithmic hyperballeans is much more interesting. In particular, we compute it for some well known groups, such as $\Z$ ($\asdim\LL(\Z)=\infty$) or the Pr\"uffer $p$-group $\Z_{p^{\infty}}$ ($\asdim\Z_{p^{\infty}}=1$), where $p$ is a prime, we find necessary conditions on an abelian groups $G$ that imply $\asdim\LL(G)<\infty$ ($G$ has to be torsion and layerly finite), and we characterise those groups $G$ with $\asdim\LL(G)=0$ ($G$ has to be torsion, reduced and with all $p$-ranks finite).

The last part of our study is focused on answering the following natural question. If $G$ and $H$ are two isomorphic groups, then $\mathcal L(G)$
and  $\mathcal L(H)$ are asymorphic (i.e., isomorphic in the category of balleans and coarse maps) and we write $\mathcal L(G)\aaa\mathcal L(H)$.
Moreover, the isomorphism between $G$ and $H$ yields also $\LL(G)\aaa \LL(H)$. However, the converse is not true in general. As a ``rigidity result'' we mean a set of conditions that imply that these converse implications holds. In other words, it is a collection of properties that implies that $G$ is isomorphic to $H$ whenever $\mathcal L(G)\aaa\mathcal L(H)$ or $\LL(G)\aaa\LL(H)$. 
Note that this is not the usual notion of rigidity in large-scale geometry (see, for example, \cite{Roe}). In particular, we focus on some special cases, namely, for a group $G$, we investigate the hypothesis $\mathcal L(G)\aaa\mathcal L(\Z)$, $\mathcal L(G)\aaa\mathcal L(\Z_{p^{\infty}})$, $\LL(G)\aaa\LL(\Z)$, and $\LL(G)\aaa\LL(\Z_{p^{\infty}})$, for some $p$ prime, and we obtain the following:
\begin{theoremintro}\label{theointro:exp}
Let $G$ be a group.
\begin{enumerate}[$\bullet$]
\item Suppose that $G$ has an element of infinite order. Then $\mathcal L(G)\aaa\mathcal L(\Z)$ if and only if $G\simeq\Z$.
\item Suppose that $G$ is abelian, then $\mathcal L(G)\aaa\mathcal L(\Z)$ if and only if either $G\simeq\Z$ or $G\simeq\Z_{p^{\infty}}$, for some prime $p$.
\end{enumerate}
\end{theoremintro}

\begin{theoremintro}\label{theointro:log}
Let $G$ be a group and $p$ be a prime.
\begin{enumerate}[$\bullet$]
\item $\LL(G)\aaa\LL(\mathbb{Z})$ if and only if $G\simeq \mathbb{Z}$;
\item $\LL(G)\aaa\LL(\mathbb{Z}_{p^{\infty}})$ if and only if $G\simeq \mathbb{Z}_{q^{\infty}}$ for some prime $q$.
\end{enumerate}
\end{theoremintro}

The paper is organised as follows. In Section \ref{sec:basic} we provide the necessary background in coarse geometry and then, in Section \ref{sec:hyper}, both the exponential hyperballean, $\exp\BB$, (\S\ref{sub:hyper}) and the logarithmic hyperballean, $\ell\mbox{-}\exp\BB$, (\S\ref{sub:log}) are introduced. In particular, we prove in \S\ref{sub:hyper} that the exponential ballean of a cellular ballean is cellular. Section \ref{sec:connected_comp} is devoted to the subgroup exponential hyperballean and the subgroup logarithmic hyperballean. In the first part of this section, we discuss a characterisation of the subgroup logarithmic hyperballean and the connected components of the balleans $\mathcal L(G)$ and $\LL(G)$. In particular, in \S\ref{sub:isoLG} we show how the existence of isolated points is related to divisibility. Furthermore, \S\ref{sub:sub_exp} is devoted to the study of the subgroup exponential hyperballean, while \S\ref{sec:logL} is focused on the subgroup logarithmic hyperballean, with a particular emphasis on its asymptotic dimension. In Section \ref{sub:G-space} we present and briefly discuss another ballean structure, $G\mbox{-}\exp\BB_G$, on the power set of a group $G$, namely, the one induced by the action of $G$ on its power set by left shifts. In particular we show that, if $G$ is infinite, the number of connected components of $\exp\BB_G$, $\ell\mbox{-}\exp\BB_G$, and $G\mbox{-}\exp\BB_G$ coincides and it is $2^{\lvert G\rvert}$.
In Section \ref{sec:rigidity} we prove our rigidity results. More in detail, in \S\ref{sub:rig_sub_exp} we show the ones concerning the subgroup exponential hyperballean (e.g., Theorem \ref{theointro:exp}), while in \S\ref{sub:rig_sub_log} we consider the subgroup logarithmic hyperballean (Theorem \ref{theointro:log} is proved there). In \S\ref{sub:rig_div} we discuss some similar results for divisible groups and in \S\ref{sub:rig_ce} we study the case where the subgroup exponential hyperballean of a group is coarsely equivalent to the one of $\Z$ or $\Z_{p^{\infty}}$, for some prime $p$.

For further progress on the connection of properties of (functrorial) coarse structures on abelian (topological) groups and their
algebraic (resp., topological) structure see the forthcoming papers \cite{DikZa1,DikZa2}.

	

\subsection*{Notation and terminology}

In the sequel, we adopt the standard notation in group theory and uniform space, following \cite{F,b4} and \cite{Isb}, respectively.
More specifically, we denote by $\N$, $\Z$, $\Q$, $\R$ and $\Z(m)$ the set of naturals, the group of integers,  the group of rationals, the group of
real and the cyclic group of order $m$, respectively. The circle group $\T = \R/\Z$ will be denoted additively, as well as its
subgroup $\Z_{p^{\infty}}=\langle\{1/p^n\in \T:  n\in\N\}\rangle$, for a prime $p$, known as the {\em Pr\" uffer $p$-group}.

Furthermore, we denote by $0$ the identity of an abelian group $G$, by $t(G)$ its torsion subgroup, and by $r_0(G)$ the free rank of $G$, i.e., the cardinality of the maximal independent subset of $G$ ($r_0(G)$ can be also defined as $\dim_{\Q} G/t(G) \otimes \Q$).
For $m\in\N$, set
$$
G[m]:=\{mx :  x\in G\}\mbox{ and }mG = \{mx: x\in G\},
$$
they are subgroup of $G$. We say that $G$ is {\em divisible} if $mG = G$ for every $m\in\N\setminus\{0\}$.
Every abelian group $G$ has a largest divisible subgroup, denote by $d(G)$.

 For prime $p$ let $r_p(G)$ be the $p$-rank of $G$ (defined as $\dim_{\Z/p\Z}G[p]$) and $r(G) = r_0(G) + \sum_p r_p(G)$.
  Finally, we denote by $\pi(G)$ the set of primes $\{p: r_p(G) > 0\}$.

 For a set $X$, we denote by $\mathcal P(X)$ its power set and we let $\exp X=\mathcal P(X)\setminus\{\emptyset\}$.

\section{Basic definitions and general constructions}\label{sec:basic}

A {\em ballean} is a triple $\BBB=(X,P,B)$ where $X$ and $P$ are sets, $P\neq\emptyset$,  and $B\colon X\times P \to \mathcal P(X)$ is a map, with the following properties:
\begin{enumerate}[$\bullet$]
	\item $x\in B(x,\alpha)$ for every $x\in X$ and every $\alpha\in P$;
	\item {\em symmetry}, i.e., for any $\alpha\in P$ and every pair of points $x,y\in X$, $x\in B(y,\alpha)$ if and only if $y\in B(x,\alpha)$;
	\item {\em upper multiplicativity}, i.e., for any $\alpha,\beta\in P$, there exists a $\gamma\in P$ such that, for every $x\in X$, $B(B(x,\alpha),\beta)\subseteq B(x,\gamma)$, where $B(A,\delta)=\bigcup\{B(y,\delta): y\in A\}$, for every $A\subseteq X$ and $\delta\in P$.
\end{enumerate}
The set $X$ is called {\em support of the ballean}, $P$ -- {\em set of radii}, and $B(x,\alpha)$ -- {\em ball of centre $x$ and radius $\alpha$}. When the ballean structure $(X,P,B)$ we are considering on $X$ is clear, we often denote $(X,P,B)$ by its support $X$.

This definition of ballean does not coincide with, but it is equivalent to the usual one (see \cite{b3} for details).

A ballean $\mathcal{B}$   is called {\em connected} if, for any $x,y \in X$, there exists $\alpha\in P$  such that $y\in B(x,\alpha)$.
Every ballean $(X,P,B)$ can be partitioned in its {\em connected components}: the {\em connected component} of a point $x\in X$ is
$$
\QQ_X(x)=\bigcup_{\alpha\in P}B(x,\alpha).
$$
Call a point $x$ of such a space $X$ {\em isolated}, if $\{x\}$ is the only ball
centred at $x$, i.e., when the connected component of $x$ is $\{x\}$. Moreover, denote by $\mathcal Iso(X)$ the family of all isolated points of the ballean $X$.

We call a subset $A$ of a ballean $(X,P,B)$ {\em bounded} if there exists $\alpha\in P$ such that, for every $y\in A$, $A\subseteq B(y,\alpha)$. The empty set is always bounded. A ballean is {\em bounded} if its support is bounded. In particular, a bounded ballean is connected.

If $(X,P,B)$ is a ballean, a subset $A$ of $X$ is {\em large in $X$} if there exists $\alpha\in P$ such that $B(A,\alpha)=X$.

If $\BBB=(X,P,B)$ is a ballean and $Y$ a subset of $X$, one can define the {\em subballean $\BBB|_Y=(Y,P,B_Y)$ on $Y$ induced by $\BBB$}, where $B_Y(y,\alpha)=B(y,\alpha)\cap Y$, for every $y\in Y$ and $\alpha\in P$.

Let $f,g\colon S\to(X,P,B)$ be two maps from a set to a ballean. Then $f$ and $g$ are {\em close} (and we write $f\sim g$) if there exists $\alpha\in P$ such that $f(x)\in B(g(x),\alpha)$, for every $x\in S$. Let $\BB_X=(X,P_X,B_X)$ and $\BB_Y=(Y,P_Y,B_Y)$ be two balleans. Then a map $f\colon X\to Y$ is called
\begin{enumerate}
	[$\bullet$]
	\item {\em bornologous} if for every radius $\alpha\in P_X$ there exists a radius $\beta\in P_Y$ such that $f(B_X(x,\alpha))\subseteq B_Y(f(x),\beta)$ for every point $x\in X$;
	\item {\em effectively proper} if for every $\alpha\in P_Y$ there exists a radius $\beta\in P_X$ such that $f^{-1}(B_Y(f(x),\alpha))\subseteq B_X(x,\beta)$ for every $x\in X$;
	\item an {\em asymorphism} if it is bijective and both $f$ and $f^{-1}$ are  bornologous (in this case, $\BB_X$ and $\BB_Y$ are {\em asymorphic} and we write $\BB_X\aaa\BB_Y$);
	\item a {\em coarse equivalence} if it is  bornologous and there exists another  bornologous map $g\colon Y\to X$ such that $f\circ g\sim id_Y$ and $g\circ f\sim id_X$, or, equivalently, if it is  bornologous, effectively proper and $f(X)$ is large in $\BBB_Y$ (in this case, $\BB_X$ and $\BB_Y$ are {\em coarsely equivalent}).
\end{enumerate}

Let $\BB$ and $\BB^\prime$ be two ballean structures on the same support $X$. We say that $\BB$ is {\em finer than} $\BB^\prime$ (and we write $\BB\prec\BB^\prime$) if $id_X\colon\BB\to\BB^\prime$ is  bornologous. Moreover, we identify $\BB$ and $\BB^\prime$ ($\BB=\BB^\prime$) if $\BB\prec\BB^\prime$ and $\BB^\prime\prec\BB$, i.e., $id_X\colon\BB\to\BB^\prime$ is an asymorphism. If $\BB$ is a bounded ballean on a set $X$, then, for every other ballean $\BB^\prime$ on the same support, $\BB^\prime\prec\BB$. Moreover, the {\em discrete ballean} $\BB_{dis}$ (i.e., the ballean $(X,\{\ast\},B_{dis})$ such that, for every point $x\in X$, $B_{dis}(x,\ast)=\{x\}$) is the finest ballean on $X$.

Denote by $\dsc(\BB)$ the number of connected components of the ballean $\BB$. Note that, if $\BB^\prime$ is another ballean on the same support, $\dsc(\BB^\prime)\leq\dsc(\BB)$ provided that $\BB\prec\BB^\prime$.

Let us now recall two very important examples of balleans.
\begin{example}\label{ex:ball}
\begin{enumerate}[(a)]
\item An {\em extended-metric}
on a set $X$ is a metric that can take also the value $\infty$. For every extended-metric space $(X,d)$, define the {\em metric ballean structure} $(X,\R_{\ge 0},B_d)$ as follows: for every $x\in X$ and $R\ge 0$, $B_d(x,R)$ is the usual closed metric ball centred in $x$ with radius $R$.
\item Let $X$ be a set and $\II$ be an ideal on $X$, i.e., a family of subsets of $X$ which is closed under taking finite unions and subsets. Then we define the {\em ideal ballean}, $\BB_{\II}=(X,\II,B_\II)$, on $X$, where, for every $x\in X$ and $K\in\II$,
$$
B_\II(x,K)=\begin{cases}
\begin{aligned}& K& \text{if $x\in K$,}\\
& \{x\} &\text{otherwise.}\end{aligned}
\end{cases}
$$
\item Let $G$ be a group and $\II$ be an ideal on $G$. Then $\II$ is a {\em group ideal} on $G$ if $\{e\}\in\II$ and, for every $F,K\in\II$, $F^{-1}\in\II$ and $FK\in\II$. For an abelian group $G$ we mostly use additive notation.

If $\kappa$ is an infinite regular cardinal, then $[G]^{<\kappa}=\{F\subseteq G:\lvert F\rvert<\kappa\}$ is a group ideal on the group $G$. The most relevant example is $[G]^{<\omega}$, the family of all finite subsets of $G$.

For every group ideal $\II$ of a group $G$, we define the {\em group ballean},  $\BB^G_{\II}=(X,\II,B^G_\II)$, where $B^G_\II(x,A)= \{x\}  \cup Ax \cup A^{-1}x$ for every $x\in X$ and $A\in\II$.
%

We denote the ballean $\BB^G_{[G]^{<\omega}}$ shortly by $\BB_G$ and we call it {\em finitary ballean} of $G$.

Groups endowed with balleans induced by group ideals have the following remarkable property.

Let $f\colon G\to H$ be a homomorphism between two groups and let $\II$ and $\mathcal J$ be two group ideals on $G$ and $H$ respectively. Then:
\begin{enumerate}[$\bullet$]
\item $f\colon\BB_{\II}^G\to\BB_{\mathcal J}^H$ is  bornologous if and only if $f(\II)=\{f(I): I\in\II\}\subseteq\mathcal J$;
\item $f\colon\BB_{\II}^G\to\BB_{\mathcal J}^H$ is effectively proper if and only if $f^{-1}(\mathcal J)=\{f^{-1}(J): J\in\mathcal J\}\subseteq\II$.
\end{enumerate}
In particular, $f\colon(G,\BB_G)\to(H,\BB_H)$ is  bornologous whenever $f$ is a homomorphism. Moreover, $f\colon(G,\BB_G)\to(H,\BB_H)$ is an asymorphism whenever $f$ is an isomorphism.
\end{enumerate}
\end{example}

Let us recall the definition of the asymptotic dimension of a ballean (see \cite{DraBel} for a comprehensive survey). A ballean $\BB=(X,P,B)$ has {\em asymptotic dimension at most $n$}, and we write $\asdim\BB\leq n$, if, for every $\alpha\in P$ there exists a uniformly bounded cover $\mathcal U=\mathcal U_0\cup\cdots\cup\mathcal U_n$ (i.e., $\bigcup\mathcal U=X$ and there exists a radius $\gamma\in P$ such that $U\subseteq B(x,\gamma)$ for every $U\in\mathcal U$ and every $x\in U$) such that, for every $i=0,\dots,n$ and every $U,V\in\mathcal U_i$, $B(U,\alpha)\cap V=\emptyset$.  For example, every bounded coarse space $X$ satisfies $\asdim X=0$, while, for every $n\in\N$, $\asdim\N^n=\asdim\Z^n=n$, where both spaces are equipped with their metric  ballean structure. The asymptotic dimension is a {\em coarse invariant}, i.e., invariant under coarse equivalence. Another important property of the asymptotic dimension is the following: if $X$ is a ballean and $Y$ a subballean of $X$, then $\asdim Y\leq\asdim X$.

 Let us now present an example of an infinity-dimensional ballean (see Proposition \ref{dimHamm}) which will be useful later in this paper. For an infinite cardinal $\kappa$, we consider the Hamming space
$$
\mathbb{H}_{\kappa}=\{f\in\{0,1\}^{\kappa} :\lvert\supt f\rvert<\omega\}
$$
endowed with the metric $h(f,g)=\lvert\supt f  \triangle\supt g\rvert$, for every $f,g\in\mathbb H_\kappa$, and denote  by $\mathbb{H} _{\kappa}^{\ast}$ the space $\mathbb{H} _{\kappa}$ with deleted the zero function.
	In the sequel, we identify $\mathbb H_\kappa$ with the set $[\kappa]^{<\omega}$, where every function $f\in\mathbb H_\kappa$ can be identified with its support.

\begin{proposition}\label{dimHamm} For every infinite cardinal $\kappa$, $\asdim\mathbb{H}_{\kappa}=\infty$.   \end{proposition}

\begin{proof} It is enough to check that $\asdim\mathbb H_\omega=\infty$ . To see that  $\asdim\mathbb{H}_{\omega}=\infty$, we take an arbitrary  $n\in \mathbb{N}$, and we claim that there exists a copy of $\N^n$ in $\mathbb H_\kappa$. That shows $\asdim\mathbb H_\omega\ge n$. Let $\omega=W^{1}\cup  \cdots \cup W^{n}$ be a partition of $\omega$ in infinite subsets. Enumerate $W^{i}= \{a_m^i: m\in\N\}$ and define, for every $i=1,\dots,n$ and every $m\in\N$, $A_m^i=\{a_0^i,\dots,a^i_m\}$.
	
	We define a map $\varphi\colon\N^n\to\mathbb{H}_\kappa$ as follows: for every $(m_1,\dots,m_n)\in\N^n$,
	$$
	\varphi(m_1,\dots,m_n)=\bigcup_{i=1}^nA_{m_i}^i.
	$$
	We claim that $\varphi$ is an isometry onto its image. In fact, for every $(m_1,\dots,m_n),(m_1^\prime,\dots,m_n^\prime)\in\N^n$,
	$$
	h(\varphi(m_1,\dots,m_n),\varphi(m_1^\prime,\dots,m_n^\prime))=\sum_{i=1}^n\lvert A_{m_i}^i\triangle A_{m_i^\prime}^i\rvert=\sum_{i=1}^n\lvert m_i-m_i^\prime\rvert=d((m_1,\dots,m_n),(m_1^\prime,\dots,m_n^\prime)),
	$$
	which shows the desired property.
	%
\end{proof}

\subsection{Some categorical constructions}\label{sub:cat}

Let $\{\BBB_i=(X_i,P_i,B_i)\}_{i\in I}$ be a family of balleans.  Let $X=\Pi_iX_i$ and $p_i\colon X\to X_i$, for every $i\in I$, be the projection maps. We denote the subset $\bigcap_{i\in I}p_i^{-1}(A_i)$ of $X$ by $\Pi_{i\in I}A_i$, where $A_i\subseteq X_i$, for every index $i\in I$. The {\em product ballean structure} on $X$ can be described as follows: this is the ballean structure $\Pi_{i\in I}\BBB_i=(X,\Pi_{i\in I}P_i,B_X)$, where, for each $(x_i)_{i\in I}\in X$ and each $(\alpha_i)_i\in\Pi_iP_i$, we put
$$
B_X((x_i)_i,(\alpha_i)_i)=\Pi_{i\in I}B_i(x_i,\alpha_i).
$$

Let $\{\BB_i=(X_i,P_i,B_i)\}_{i\in I}$ be a family of balleans. Define their {\em coproduct} as the ballean $\coprod_i\BB_i=(X,P,B)$, where $X=\bigsqcup_iX_i$, $P=\{(\alpha_k)_{k\in K}: K\in[I]^{<\omega},\,\alpha_k\in P_k,\,\forall k\in K\}$, and
$$
B(i_j(x_j),(\alpha_k)_{k\in K})=\begin{cases}
\begin{aligned}& i_j(B_j(x_j,\alpha_j))&\text{if $j\in K$,}\\
& i_j(\{x_j\}) &\text{otherwise,}\end{aligned}
\end{cases}
$$
where $i_k\colon X_k\to X$ is the canonical inclusion, $i_j(x_j)\in X$, and $(\alpha_k)_{k\in K}\in P$.

If $X$ is a ballean with finitely many connected components, then $X$ coincides with the coproduct of its connected components. This may fail when
$X$ has infinitely many connected components, see Example \ref{ex:connected_components}.

\begin{remark}\label{RemIso}
	If $\varphi\colon X \to Y$ is an asymorphism, then $\varphi$ retracted to $\mathcal{I}so(X)$ is a bijections between $\mathcal{I}so(X)$ and $\mathcal{I}so(Y)$ (so $X\aaa Y$ trivially yields $|\mathcal{I}so(X)|=| \mathcal{I}so(Y)|$). Moreover, $\varphi$ determines a bijection between the family of non-trivial connected components  of $X$ and its counterpart in $Y$, so one can index both families with the same index set $I$ and write  $X \setminus \mathcal{I}so(X) = \bigcup_{i\in I}{\mathcal Q}_X(x_i)$ and $Y \setminus \mathcal{I}so(Y) = \bigcup_{i\in I}{\mathcal Q}_Y(y_i)$, assuming without loss of generality that $\varphi(x_i) = y_i$ and the restriction of $\varphi$ determines asymorphisms ${\mathcal Q}_X(x_i)\aaa {\mathcal Q}_Y(y_i)$. All these are only {\em necessary conditions} for $X\aaa Y$ (i.e., the bare fact that  $|\mathcal{I}so(X)|=| \mathcal{I}so(Y)|=:|I|$ and ${\mathcal Q}_X(x_i)\aaa {\mathcal Q}_Y(y_i)$ for all $i \in I$ need not imply $X\aaa Y$, see Example \ref{ex:connected_components}). Indeed, $X\aaa Y$ imposes a ``uniform coarseness" of the asymorphisms ${\mathcal Q}_X(x_i)\aaa {\mathcal Q}_Y(y_i)$.
	
	Moreover, if $I$ is finite, then they are indeed sufficient conditions, since, in that case, both $X$ and $Y$ coincide with the coproducts of their connected components.
\end{remark}


\begin{fact}\label{fact:ce_bounded_unbounded}
	Let $\varphi\colon X\to Y$ be a coarse equivalence. Then $\dsc X=\dsc Y$. Moreover, for every $x\in X$, $\QQ_X(x)$ is bounded if and only if $\QQ_Y(\varphi(x))$ is bounded.
\end{fact}



\subsection{Some examples of ballean classes: thin and cellular balleans}\label{suc:cell}

Let $\BB=(X,P,B)$ be a ballean. A subset $Y$ of $X$ is {\em thin} if, for every $\alpha\in P$, there exists a bounded subset $V\subseteq X$ such that $\lvert B(x,\alpha)\cap Y\rvert=1$, for every $x\in Y\setminus V$. A ballean is thin if its whole support is thin.

Let $\mathcal{I}$  be an ideal on $X$. The ideal ballean $\BB_{\II}=(X,\II,B_\II)$ is connected if and only if $\II$ is a cover or, equivalently, $[X]^{<\omega}\subseteq\II$.

\begin{fact}\label{fact:asy_ideal}
Let $X$ and $Y$ be two non-empty sets, $\II$ and $\mathcal J$ be two ideals of $X$ and $Y$, respectively, and $f\colon X\to Y$ be a map. Then $f$ is an asymorphism between $\BB_\II$ and $\BB_{\mathcal J}$ if and only if $f$ is a bijection such that $f(\II)\subseteq\mathcal J$ and $f^{-1}(\mathcal J)\subseteq\II$, for every $I\in \II$ and $J\in\mathcal J$.
\end{fact}
\begin{corollary}
Let $X$ and $Y$ be two sets endowed with the ideal ballean structures induced by finite subsets. Then $X$ and $Y$ are asymorphic if and only if $\lvert X\rvert =\lvert Y\rvert$.
\end{corollary}

Among all characterisations of thinness (for example, see \cite{b5,b3}), let us remind the following.
\begin{proposition}\label{prop:charact_thin}
A connected ballean $\BB=(X,P,B)$ is thin if and only if $\BB=\BB_{\II}$, where $\II$ is the ideal of bounded subsets of $\BB$.
\end{proposition}
The notion of thinnes may seem too restrictive. In fact, one can easily see that a non-connected ballean is thin if and only if all but one connected components are trivial and the non-trivial one is thin. Since the balleans we are considering in this paper are non-connected, we define the following class. A ballean $\BB$ is {\em weakly thin} if every connected component is thin.

An important class of balleans is defined as follows.  For a ballean $(X,P,B)$, $n\in \N$, $x\in X$ and $\alpha\in P$, we let
$$
B^n(x,\alpha)=\underbrace{B(B(\cdots B(B}_{\text{$n$ times}}(x,\alpha),\alpha)\cdots,\alpha),\alpha) \ \  \mbox{ and }\ \  B^{\square}(x,\alpha):=\bigcup_{n=1}^\infty B^n(x,\alpha).
$$

\begin{definition}\label{def_cell} The triple $\BBB^{\square}=(X,P,B^{\square})$ is a ballean called the {\em cellularization of $\BBB$}.
%
%
%
The ballean $\BBB$ is said to be {\em cellular} if $\BBB=\BBB^{\square}$.
\end{definition}

Cellular balleans are precisely those with asymptotic dimension 0. Thin balleans and, in particular, bounded balleans are cellular. Moreover, if a weakly thin ballean has only a finite number of connected components, it is cellular. However, this property doesn't hold for weakly thin balleans with infinite number of connected components (see Example \ref{ex:connected_components}).

\begin{example}\label{ex:connected_components} Consider the ballean $X=\bigsqcup_{n\in\N}[0,n]$, endowed with the natural extended metric defined as follows: if $x,y$ belong to the same component $I_n :=[0,n]$, then $d(x,y) = |x-y|$, if $x,y$ belong to distinct components, we put $d(x,y) =\infty$.
 Then $\asdim X>0$ and thus it is not cellular, although it is weakly thin. Finally, the ballean $Y=\coprod_n[0,n]$ is cellular and thus $X$ and $Y$ are not coarsely equivalent (and, in particular, not asymorphic).
\end{example}

\section{The exponential and the logarithmic hyperballeans}\label{sec:hyper}

\subsection{The exponential hyperballean $\exp X$}\label{sub:hyper}

Recall the following definition from \cite{b5}.


\begin{definition}\label{1.1}
 For a  ballean $\mathcal{B}=(X, P, B)$, we  consider the {\it hyperballean} $\exp\mathcal{B} = (\exp(X), P, \exp B)$, where
	$$
	\exp B(Y, \alpha )= \{Z\in\exp X: Z\subseteq B(Y, \alpha ), \  Y\subseteq B(Z, \alpha )\},
	$$
for every $Y\in\exp X$ and every $\alpha\in P$.
\end{definition}

Note that, if $X$ and $Y$ are two asymorphic balleans, then $\exp X\aaa\exp Y$.

We use also the subballean $\mathcal{B}^{\flat}=(X^{\flat}, P, B^{\flat})$ of  $\exp\mathcal{B}$, where $X^{\flat}$ is the family of all non-empty bounded subsets of $X$ and $B^{\flat}$ is the restriction of $\exp B$ to $X^{\flat}$. This ballean was already defined in \cite{ProPro}.

It is trivial that, for every ballean $X$,
\begin{equation}\label{pnt:vs:hyper}
\asdim X\leq\asdim\exp X
\end{equation}
(in fact the subballean of $\exp X$ whose support is the family of all singletons of $X$ is asymorphic to $X$, according to \cite{b5}). The equality $\asdim X=\asdim\exp X$ is not available in general and may strongly fail (see \cite{rs}).
However, that equality is true for asymptotic dimension $0$, as we will show in Proposition \ref{prop:exp_cellular}.
\begin{lemma}\label{lemma:cellular_exp}
Let $(X,P,B)$ be a ballean. Let $\emptyset\neq Y\subseteq X$ and $\alpha$ be an arbitrary element of $\exp X$ and of $P$, respectively. Then
\begin{equation}\label{eq:cell}
(\exp B)^n(Y,\alpha)\subseteq\{Z\in\exp X : Z\subseteq B^n(Y,\alpha),\,Y\subseteq B^n(Z,\alpha)\}.
\end{equation}
Hence, in particular,
\begin{equation*}
(\exp B)^\square(Y,\alpha)\subseteq\{Z\in\exp X : Z\subseteq B^\square(Y,\alpha),\,Y\subseteq B^\square(Z,\alpha)\}=\exp(B^\square)(Y,\alpha).
\end{equation*}
\end{lemma}

\begin{proof} We prove \eqref{eq:cell} by induction. The base step $n=1$ is trivial.
%
Suppose that \eqref{eq:cell} holds for some $n\geq $ and let $Z\in\exp B^{n+1}(Y,\alpha)$. There exists $W\in\exp B(Y,\alpha)$ such that $Z\in\exp B^n(W,\alpha)$ and so, by using the inductive hypothesis,
$$
W\subseteq B(Y,\alpha),\quad Y\subseteq B(W,\alpha),\quad Z\subseteq B^n(W,\alpha),\quad W\subseteq B^n(Z,\alpha),
$$
from which the claim follows.
\end{proof}

\begin{proposition}\label{prop:exp_cellular}
If a ballean $(X,P,B)$ is cellular, then $\exp X$ is cellular.
\end{proposition}
\begin{proof}
Let $\alpha$ be an arbitrary radius and let $\beta\in P$ be a radius such that $B^\square(x,\alpha)\subseteq B(x,\beta)$, for every $x\in X$. Then, by using Lemma \ref{lemma:cellular_exp}, for every $Y\in\exp X$,
\begin{align*}
(\exp B)^\square(Y,\alpha)&\,\subseteq\{Z\in\exp X :  Y\subseteq B^\square(Z,\alpha),\,Z\subseteq B^\square(Y,\alpha)\}\\
&\,\subseteq\{Z\in\exp X : Y\subseteq B(Z,\beta),\,Z\subseteq B(Y,\beta)\}=\exp B(Y,\beta).
\end{align*}
\end{proof}

In this paper we are mostly interested in some special hyperballeans, namely $\exp\BB_G$, where $G$ is a group. The following trivial fact will be used in the sequel.

\begin{fact}\label{fact:finite_balls}
Let $G$ be a group and $e$ be its identity. Then every ball of $\exp\BB_G$ centred in $\{e\}$ is finite. Hence, for every subset $X$ of $G$ such that $X\in\QQ_{\exp\BB_G}(\{e\})$ and every finite subset $F$ of $G$, the ball $\exp B_G(X,F)$ is finite.
\end{fact}
\begin{proof}
It is enough to note that, for every finite subset $F$ of $G$, if $X\in\exp B_G(\{e\},F)$, then $X\subseteq \{e\} \cup F$. The second statement follows, since $\exp\BB_G$ is upper multiplicative.
\end{proof}

\subsection{The logarithmic hyperballean $\ell\mbox{-}\exp G$}\label{sub:log}

Let us now introduce another ballean structure on $\exp(G)$, where $G$ is a group.

\begin{definition}\label{3.1}
For a group $G$, we define  a function $d\colon\exp(G)\times\exp(G)\to\N\cup\{\infty\}$  as  follows. If $Y, Z\subseteq G$ are two non-empty subsets which are in distinct connected components of $ \exp\mathcal{B}_{G}$ then $d(Y, Z)=\infty$.  Otherwise, we define
$$
\mu (Y, Z)= \min\{\max\{\lvert F\vert, \lvert S\rvert\}: FY \supseteq Z,\, SZ\supseteq Y,\, F\in[G]^{<\omega},\, S \in[G]^{<\omega}, \,  e\in F\cap S\}
$$
and put
	$$d(Y,Z)=\log\mu (Y,Z),$$
 where the base of the logarithm is any value strictly greater than $1$.
\end{definition}

The next claim is not hard to check, yet we give an argument for the sake of completeness.
\begin{claim} The function $d$ is an extended metric. Moreover, changing the base of the logarithm leads to equivalent extended metrics.
 \end{claim}
\begin{proof}	 In fact, the only non-trivial property is the triangular inequality. Fix then three non-empty subsets $X,Y,Z$ of $G$ such that $d(X,Y)\leq\log n$ and $d(Y,Z)\leq\log m$. Pick four finite subsets $F_1,F_2,K_1,K_2\subseteq G$ such that
\begin{enumerate}[(a)]
\item $e\in F_1\cap F_2\cap K_1\cap K_2$,
\item $\lvert F_1\rvert,\lvert F_2\rvert\leq n$,
\item $\lvert K_1\rvert,\lvert K_2\rvert\leq m$,
\item $X\subseteq F_1Y$ and $Y\subseteq F_2X$,
\item $Y\subseteq K_1Z$ and $Z\subseteq K_2Y$.
\end{enumerate}
In particular $X\subseteq F_1Y\subseteq F_1K_1Z$ and, similarly, $Z\subseteq K_2F_2 Y$. Since both $\lvert F_1K_1\rvert\leq mn$ and $\lvert F_2K_2\rvert\leq mn$, this proves the first part of the claim. The last part is trivial.
\end{proof}	

Finally we define the {\it logarithmic hyperballean} as the metric ballean induced by $d$, namely
$$
\ell\mbox{-}\exp\mathcal{B}_{G} = (\mathcal P(G)\setminus\{\emptyset\},\mathbb{R}_{\ge 0}, B_{d} ),\text{ where $B_{d}(Y,R)=\{Z: d(Y,Z)\leq R\}$, for every $\emptyset\neq Y\subseteq G$ and $R\ge 0$}.
$$

The extended metric $d$  is invariant under left and right actions of $G$  on $\mathcal P(G)\setminus\{\emptyset\}$, i.e., the maps  $F\mapsto gF$ and $F\mapsto Fg$, for every $g\in G$ and every $F\subseteq G$.

Furthermore, if $G$ and $H$ are two isomorphic groups, then $\LL(G)\aaa\LL(H)$.

\begin{remark}\label{rem:logarithmic}
\begin{enumerate}[(a)]
\item Clearly, the connected components of $\exp\mathcal{B} _{G}$ and $\ell\mbox{-}\exp\BB_G$ coincide.
\item For every group $G$, $\exp\BB_G$ is finer than $\ell\mbox{-}\exp\BB_G$. In fact, if two subsets $X$ and $Y$ of $G$ satisfy $X\in\exp B_G(Y,F)$ for some finite subset $F\in\mathcal F_G$, then $d(X,Y)\leq\log\lvert F\cup \{e\}\rvert$.
\item If $G$  is infinite, then $\exp\mathcal{B}_{G}$ is strictly finer than $\ell\mbox{-}\exp\mathcal{B}_{G}$. First of all, note that, for every two distinct singletons $\{x\}$ and $\{y\}$ of $G$, $d(\{x\},\{y\})=1$ and thus $B_d(\{x\},1)\supseteq\{\{z\}: z\in G\}$. However, a subset $K$ of $G$ such that $\{x\}\subseteq\{y\}K$, for every $x,y\in G$, must satisfy $K=G$, which is not a radius of $\exp\BB_G$, since $G$ is infinite.
\item 
If $G$ is abelian, then the $\ell\mbox{-}\exp\mathcal{B}_{G} $ can also be defined by the extended-metric $d^{\prime}$, where, for every two non-empty subsets $Y,Z\subseteq G$,
$$ d^{\prime}(Y,Z)= \log \min \{\lvert S\rvert: S+Y \supseteq Z,\, S+Z\supseteq Y,\, S\in [G]^{<\omega},\, 0\in S  \},$$
if $Y, Z$  are in the same connected component of $\exp\mathcal{B}_{G}$, and $d^{\prime}(Y,Z)=\infty$, otherwise.
\end{enumerate}
\end{remark}

\begin{remark}
Let $G$ be a group of cardinality $\kappa$. Recall that $G^\flat$ is the family of all non-empty bounded subset of $(G,\BB_G)$, i.e., all finite subsets of $G$. We consider two ballean structure on $G^\flat$: the first one is the subballean structure $\ell\mbox{-}\BB_G^\flat=\ell\mbox{-}\exp\mathcal{B}_{G}|_{G^\flat}$, while the second one is given by the identification of $\flat(G)$ with $\mathbb H_\kappa^\ast$, i.e., the metric ballean induced by $h$, where, for every $X,Y\in G^\flat$, $h(X,Y)=\lvert X\triangle Y\rvert$.

We claim that $\mathbb H_\kappa^\ast$ is finer than $\ell\mbox{-}\BB_G^\flat$ and, moreover, if $G$ is infinite, it is strictly finer. Let $R\in\N$ and $X,Y\in G^\flat$ such that $h(X,Y)\leq R$. Fix two elements $\overline x\in X$ and $\overline y\in Y$ and define
$$
F=\{x\overline y^{-1} : x\in X\setminus Y\}\cup\{e\}\quad\mbox{and}\quad K=\{y\overline x^{-1} : y\in Y\setminus X\}\cup\{e\}.
$$
Then $\lvert F\rvert,\lvert K\rvert\leq R+1$ and $$
X\subseteq(X\setminus Y)\cup Y=F\overline y\cup Y\subseteq FY\quad\mbox{and}\quad Y\subseteq(Y\setminus X)\cup X=K\overline x\cup X\subseteq KX,
$$
which implies the first part of the statement. Suppose now that $G$ is infinite. Then, for every $n\in\N$, there exists $F\in G^\flat$ and $g\in G$ such that $F\cap gF=\emptyset$ and $|F|=n$. Then $d(F,gF)=1$, while $h(F,gF)=2n$. Since $n$ can be chosen arbitrarily, $\mathbb{H}_\kappa^\ast$ is strictly finer than $\ell\mbox{-}\BB_G^\flat$.
\end{remark}

\begin{question}\label{q:hamm}	
\begin{enumerate}[(a)]
\item For a countable group $G$, are $\ell\mbox{-}\BB_{G}^\flat$   and $ \mathbb{H}_{\omega}^{\ast} $ asymorphic? Coarsely equivalent?
\item If the answer to item (a) is affirmative then, if $G$ and $H$ are countable groups, are $\ell\mbox{-}\BB_{G}^\flat$ and $\ell\mbox{-}\BB_{H}^\flat$ asymorphic? In particular, what does it happen if $G= \mathbb{Z}$  and $H$ is the countable group of exponent 2?
\end{enumerate}
\end{question}


\section{The subgroup hyperballeans $\mathcal L(G)$ and $\LL(G)$}\label{sec:sub_hyper}
\label{sec:connected_comp}

For every group $G$, we denote by $L(G)$ the lattice of all subgroups of $G$.
In this paper we focus on the following subballeans of $\exp\BB(G)$ and $\ell\mbox{-}\exp\BB(G)$:

\begin{enumerate}[$\bullet$]
\item the {\em subgroup exponential hyperballeans} $\mathcal L(G)=\exp\BB_G|_{L(G)}$; and
\item  the {\em subgroup logarithmic hyperballeans} $\LL(G)=\ell\mbox{-}\exp\BB(G)|_{L(G)}$.
\end{enumerate}

First of all, we want to give a different and useful characterisation of the subgroup logarithmic hyperballean $\LL(G)=\ell\mbox{-}\exp\BB_G|_{L(G)}$, where $G$ is a group, and, in order to do that, we need the following result.

\begin{lemma}\label{2.3}  Let $G$ be a group and let $A, B$ be subgroups of $G$ such that $B\subseteq SA$ for some subset $S$  of $G$. Then $\lvert B: (A \cap B)\rvert\leq \lvert S\rvert$.
\end{lemma}

\begin{proof}	 We split the proof in three cases.
	\medskip

	{\bf Case 1.}  {\sl Assume that $S\subseteq B$. } Given any $b\in B$, we pick $s\in S$  such that $b\in sA$. Then $s ^{-1} b \in A \cap B$ and $B\subseteq S(A\cap B)$.
	This proves that  $|B:(A \cap B)|\leq  |S|$.
	
	\medskip
	
	{\bf Case 2.} {\sl Assume that  $S \subseteq BA$}. Let $S_a:= S\cap Ba$ and note that our assumption provides  a partition $S= \bigcup_{a\in A} S_a$. Let $S^*:= \bigcup_{a\in A, S_a\ne \emptyset } S_aa^{-1}$ and note that: \\
	(i) $SA = S^*A$; (ii) $|S^*|\leq |S|$; (iii) $S^*\subseteq B$ (as $S_aa^{-1}\subseteq B$ when $ S_a\ne \emptyset$).
	
	By (i) and our blanket assumption $B\subseteq SA$, $B\subseteq S^*A$, so by (iii) we can apply Case 1 to $A,B$ and $S^*$ to claim
	$|B: A \cap B|\leq  |S^*|$. Now (ii) allows us to conclude that $|B: (A \cap B)|\leq  |S|$.
	
	\medskip
	
	{\bf Case 3.} In the general case let	$S_1:=S \cap BA$. Then obviously, $B\subseteq S_1A$ and $S_1 \subseteq BA$. By case 2,
	applied to $A,B$ and $S_1$ we have $|B: A \cap B|\leq  |S_1|$. Since obviously $|S_1| \leq |S|$,	this 	 yields $|B: A \cap B|\leq  |S|$. \end{proof}

\begin{definition} We recall that two subgroups of a group $G$ are {\it commensurable} if the indices  $\lvert A: A \cap B\rvert$ and $\lvert B: A\cap B\rvert$ are finite.
\end{definition}

By Lemma \ref{2.3} and Remark \ref{rem:logarithmic}(a), two subgroups $A$ and $B$ of $G$ are in the same connected component of $\mathcal{L}(G)$ ($\LL(G)$, equivalently) if and only if $A$ and $B$ are commensurable.

Moreover, Lemma \ref{2.3} also implies a different characterization of $\LL(G)$, which is much more manageable. Namely, for every group $G$ and every pair of subgroups $A$ and $B$ of $G$, define
$$
d^\prime(A,B)=\begin{cases}
\begin{aligned}&\log\max\{\lvert A: A\cap B\rvert, \lvert B: A\cap B\rvert\} &\text{if $A$ and $B$ are commensurable,}\\
&\infty &\text{otherwise,}
\end{aligned}
\end{cases}
$$
which is an extended metric on $L(G)$. By Lemma \ref{2.3}, the ballean structure on $L(G)$ induced by $d^\prime$ coincides with $\LL(G)$.
%

Thanks to the previous characterisation of the subgroup logarithmic hyperballean, we can provide an example of a group $G$ such that $\LL(G)$ has some infinite ball (see Example \ref{ex:Tarskii}). 

\begin{remark}\label{rem:Q_Zn}
Fix $n\ge 2$. We want to take a closer look at the structure of $\mathcal L(\Z^n)$. First of all note that two commensurable subgroups $H$ and $K$ of $\Z^n$ have same free rank. Moreover, every subgroup $H$ of $\Z^n$ is commensurable with a pure subgroup $\sat(H)$ of $\Z^n$,  namely its {\em saturation} defined by
$$
\sat(H) = \{x\in \Z^n: mx \in H\  \mbox{ for some non-zero }\ m\in \Z \}
$$
(denoted also by $H_*$ by some authors; recall that a subgroup $H$ of an abelian group $G$ is {\em pure}, whenever
$mH = mG \cap H$ for every $m>0$ [pure subgroups of $\Z^n$ split as direct summands]). For every $H,K\leq\Z^n$, $\sat(H)$ is commensurable with $\sat(K)$ if and only if $\sat(H)=\sat(K)$. Then $\mathcal L(\Z^n)$ has a countable number of connected components. Namely, they are:
\begin{itemize}
	\item $\QQ_{\mathcal L(\Z^n)}(\{0\})=\{0\}$,
	\item $\QQ_{\mathcal L(\Z^n)}(\Z^n)$,
	\item for every $0<k<n$, a countable number of connected components asymorphic to the subballean $\QQ_{\mathcal L(\Z^n)}(\Z^k)$ of $\mathcal L(\Z^n)$ which is asymorphic to the subballean $\QQ_{\mathcal L(\Z^k)}(\Z^k)$ of $\mathcal L(\Z^k)$.
\end{itemize}

In particular, by Remark \ref{RemIso}, for every $n>1$, neither $\mathcal L(\Z)\aaa\mathcal L(\Z^n)$, nor $\LL(\Z)\aaa\LL(\Z^n)$.

Note that $L(\Z)$ has two connected components, while $\dsc(\exp\BB_\Z)=\dsc(\LL\BB_{\Z})=2^\omega$, as we will prove in Proposition \ref{prop:dsc_expG}.
\end{remark}

\subsection{$\iso(G)$: the chase for isolated points of $\mathcal L(G)$ and $\LL(G)$}\label{sub:isoLG}

According to Remark \ref{rem:logarithmic}(a), the isolated points of $\mathcal L(G)$ and $\LL(G)$ coincide. We denote this  set of common isolated points by $\iso(G)$.

 In this subsection we show how the existence of isolated points is related to a well-known algebraic property, such as divisibility.

 Let us recall, that  a divisible subgroup $H$ of an abelian  group $G$ always splits, i.e., there exists another subgroup $K$ of $G$ such that $G\simeq H\oplus K$.
 The class of divisible groups is stable under taking quotients, products and direct sums. Examples of divisible groups are $\Q$ and, for every prime $p$, the {\em Pr\" uffer $p$-group} $\mathbb{Z}_{p^{\infty}}$ (while non-trivial finite groups are never divisible).
As the following folklore fact shows, every divisible group is a direct sum of these. In other words,
divisible abelian groups are completely characterised by their ranks:

\begin{fact}\label{theo:divis_structure}{\rm \cite{F}}  If $G$ is a divisible abelian group, then
$$
G\simeq\Q^{(r_0(G))}\oplus\bigg(\bigoplus_{p\in\pi(G)}(\Z_{p^{\infty}})^{(r_p(G))}\bigg).
$$
\end{fact}

A group $G$ is called {\em reduced} if $d(G)=\{0\}$, in other words, the biggest divisible subgroup of $G$ is $\{0\}$.

The equivalence of (a) and (b) in the following claim is folklore, yet we give a proof for the sake of completeness.

\begin{claim}\label{Cl1} For an abelian group $G$ the following are equivalent:
\begin{enumerate}[(a)]
  \item $G$ is divisible;
  \item $G$ has no proper subgroups of finite index;
  \item $\{G\}\in \iso(G)$.
\end{enumerate}
\end{claim}

\begin{proof}  (a) $\to $ (b) It suffices to note that if $H$ is a proper subgroup of $G$ of finite index, then the quotient $G/H$ is non-trivial finite group,
so cannot be divisible, while divisibility is preserved under taking quotients.

 (b) $\to $ (a) If $G$ is not divisible, then $pG \ne G$ for some prime $p$. Hence,  $G/pG$ is a non-trivial abelian group of exponent $p$, i.e., a vector space over $\Z/p\Z$. Then $G/pG$ admits a non-zero homomorphism $G/pG\to \Z/p\Z$, which is necessarily surjective, so provides a quotient of $G$ isomorphic to $\Z/p\Z$.

 Finally, (b) and (c) are obviously equivalent.
\end{proof}

\begin{proposition}\label{abba} For a subgroup $A$ of an abelian group $G$ the following are equivalent:
\begin{enumerate}[(a)]
  \item $A\in \iso(G)$;
  \item $A$ is divisible and $G \simeq A \oplus B$, where $B$ is a torsion-free subgroup of $G$.
  \item $t(G) \leq A \leq d(G)$ and $A$ is divisible;
\end{enumerate}
\end{proposition}

\begin{proof} (b)$\to$(a) If $G = A \oplus B$, with $A$ divisible and $B$ torsion-free,  then $A$ has no proper subgroup of finite index (Claim \ref{Cl1}). So if $C$ is a subgroup of $B$ commensurable with $A$, then  $C\cap A = A$, i.e., $A \leq C$. Since $A$ is divisible, this gives  $C = A \oplus C_1$, for some subgroup $C_1$ of $B$. As  $C $ and $ A$ are commensurable, $C_1\simeq C/A$ is finite. Since $B$ is torsion-free, this yields $C_1 = \{0\}$, i.e., $C = A$. Thus, $A\in \iso(G)$.
	
(a)$\to$(b) Now assume that $A\in \iso(G)$. Then $A$ has no proper subgroups of finite index, so $A$ is divisible, by Claim \ref{Cl1}.  Then there exists a subgroup $B$ of $G$ such that $G = A \oplus B$. If $b\in B$ were a non-zero torsion element of $B$, then $A \oplus \langle b\rangle $ is commensurable with $A$, so our hypothesis implies that $\langle b\rangle = \{0\}$. This proves that $B$ is torsion-free.
	
 Finally, the equivalence (b)$\leftrightarrow$(c) is trivial.
\end{proof}

\begin{corollary}\label{coro:iso_divis} The following are equivalent for an abelian group $G$:
\begin{enumerate}[(a)]
  \item $\iso(G) \ne  \emptyset$;
  \item $G/d(G)$ is torsion-free;
  \item $t(G) \leq d(G)$;
  \item $d(G)\in  \iso(G)$.
\end{enumerate}
\end{corollary}

\begin{proof} (a) $\to$ (b) Assume that $\iso(G) \ne \emptyset$ and pick $A\in \iso(G)$. Then $G  \simeq A \oplus B$, where $B$ is a torsion-free subgroup of $G$.  As $A\leq d(G)$, one has also $ G = d(G) \oplus R(G)$, and one can  arrange to have $R(G)$ a subgroup of $B$, so $R(G)\simeq G/d(G)$ is torsion-free.
	
	 (b) $\to$ (c) $\to$ (d) $\to$(a) are obvious (the second one in view of Proposition \ref{abba}).    \end{proof}	 	

	In particular, one can easily isolate the following sufficient conditions for the (non-)existence of isolated points:
	
\begin{corollary}\label{coro:iso_divis1} For an abelian group $G$ one has:
\begin{enumerate}[(a)]
  \item $G \in \iso(G)$ $($so $ \iso(G)\ne \emptyset)$ whenever $G$ is divisible;
  \item if $G$ is reduced, then $\iso(G)=  \emptyset$ if and only if $G$ is not torsion-free; otherwise, $\iso(G)=\{\{0\}\}$ is a singleton.
\end{enumerate}
	 \end{corollary}

\begin{proof}  (a) follows from Proposition \ref{abba}(c).

(b) If $G$ is reduced, then $d(G) = \{0\}$, so $ \iso(G)\ne \emptyset$ precisely when  $t(G) = \{0\}$, according to the above corollary. The last assertion follows from Corollary \ref{coro:iso_divis}.
\end{proof}

Now we provide a sharper result that complements the previous corollaries which characterized when $\iso(G) = \emptyset$.
More precisely, we show that the size of $\iso(G)$ is completely determined by the free-rank $r_0(d(G))$ as follows:

\begin{proposition}\label{rem:rank_divisble_iso} Let $G$ be an abelian group with $\iso(G) \ne  \emptyset$ $($i.e., $t(G) \leq d(G))$. Then:
\begin{enumerate}[(a)]
	\item $\iso(G)$ has size 1 $($more precisely, $\iso(G) = \{d(G)\})$ if and only if $r_0(d(G)) = 0$, i.e., $d(G) = t(d(G))$ is torsion;
	\item $\iso(G)$ has size 2 $($more precisely, $\iso(G) = \{ t(d(G)), d(G)\})$ if and only if $r_0(d(G)) = 1$, i.e., $d(G) = t(d(G)) \oplus D$ with $D\simeq \Q$ torsion-free;
	\item $|\iso(G)| = \omega$ if and only if $1 < r_0(d(G)) = n<  \omega $ (then $d(G) = t(d(G)) \oplus D$ with $D\simeq \Q^n$); and
	\item $\iso(G)$ is uncountable $($more precisely, $|\iso(G)| =2^{r_0(d(G))})$ if and only if $r_0(d(G))\geq \omega $.
\end{enumerate}
\end{proposition}

\begin{proof} (a), (b) and (c) follow from the above corollaries and the fact that $\Q^n$ has countably many divisible subgroups when $1 < n<  \omega $. Similar arguments work for (d). \end{proof}	 	

Proposition \ref{rem:rank_divisble_iso}, along with Remark \ref{RemIso}, provides a large series of non-asymorphic pairs of spaces, like:
$$
\mathcal L(\Z^n)\not\aaa\mathcal L(\Q) \not\aaa\mathcal L(\Q^m)\  \mbox{ and \  $\mathcal L(\Z^n) \not\aaa\mathcal L(\Q\oplus \Z) \not\aaa\mathcal L(\Q^m)$ for any $n$ and $m>1$}
$$
and the same for the corresponding subgroup logarithmic hyperballeans.

In the next remark we discuss further some other immediate consequences of the above results concerning (only) the set $\iso(G)$ isolated points of an abelian group $G$ on the group structure of $G$.

\begin{remark}\label{rem:appl_iso}
\begin{enumerate}[(a)]
 \item Let $G$ be a virtually divisible abelian group and $D$ be a divisible abelian group. Then $G$ is divisible, provided that either $\mathcal L(G)\aaa\mathcal L(D)$ or $\LL(G)\aaa\LL(D)$. In fact, since $\iso(D)$ is non-empty, $G/d(G)$ must be torsion-free, by Corollary \ref{coro:iso_divis}. On the other hand, $G/d(G)$ must be finite (hence, torsion), since  $G$ be a virtually divisible. Thus, $G = d(G)$ is divisible.
 \item Let $G$ be a divisible abelian group and $H$ be a finitely generated abelian group such that
$\mathcal L(G)\aaa\mathcal L(H)$ or $\LL(G)\aaa\LL(H)$. Then $G$ is torsion and $H$ is free. In fact, assume first of all that both $G$ and $H$ are non-trivial. By Corollary \ref{coro:iso_divis1}, $ \iso(G)\ne \emptyset$, so $ \iso(H)\ne \emptyset$ as well.
Since $H$ is reduced, this fact  implies (by the same corollary) that  $H$ is torsion-free. Hence $H\simeq\Z^n$, for some $n\ge 0$, so
$\lvert\iso(H)\rvert=1$. This yileds $\lvert\iso(G)\rvert=1$. Since $G$ is divisbile, Proposition \ref{rem:rank_divisble_iso} applies to entail
that $G$ must be torsion.
 \item Let $G$ be a divisible torsion-free abelian group. Then Proposition \ref{rem:rank_divisble_iso} implies that $G\simeq\Q$, provided that $\mathcal L(G)\aaa\mathcal L(\Q)$ or $\LL(G)\aaa\LL(\Q)$.
 \item Under the assumption of the Generalised Continuum Hypothesis, if $G, D$ are divisible torsion-free abelian group
such that  at least one of them has infinite rank, then the following statements are equivalent:
\begin{itemize}
 \item[(d$_1$)] $\mathcal L(G)\aaa\mathcal L(D)$;
 \item[(d$_2$)] $\LL(G)\aaa\LL(D)$;
 \item[(d$_3$)] $|\iso(G)| = |\iso(D)|$;
 \item[(d$_4$)] $r_0(G) = r_0(D)$;
 \item[(d$_5$)] $G\simeq D$.
\end{itemize}
The implications (d$_5$) $\rightarrow$ (d$_1$) $\rightarrow$ (d$_3$) and (d$_5$) $\rightarrow$ (d$_2$) $\rightarrow$ (d$_3$) are trivial and were already discussed. Implication (d$_3$) $\rightarrow$ (d$_4$) follows from Proposition \ref{rem:rank_divisble_iso}, in particular the equality $\lvert\iso(G)\rvert =2^{r_0(d(G))}$, and GCH. The implication (d$_4$) $\rightarrow$ (d$_5$) follows from the fact that divisible torsion-free abelian groups are determined by their free-rank up to isomorphism (see Fact \ref{theo:divis_structure}).
\end{enumerate}
\end{remark}
The assertions in Remark \ref{rem:appl_iso}(c) and (d) are examples of what we call ``rigidity results'', to which Section \ref{sec:rigidity} is devoted. It is trivial that, if $G$ and $H$ are two isomorphic groups, then $\mathcal L(G)\approx\mathcal L(H)$ (more precisely, $\exp\BB_G\approx\exp\BB_H$), and $\LL(G)\aaa\LL(H)$. The converse implication is not true (Corollaries \ref{coro:Z_Zp_asy} and \ref{coro:asdim_L_N}). A rigidity result is a list of conditions on balleans $\mathcal L(G)$ and $\mathcal L(H)$ ($\LL(G)$ and $\LL(G)$), where $G$ and $H$ are two groups, which imply that $G\simeq H$, provided that $\mathcal L(G)\approx\mathcal L(H)$ ($\LL(G)\aaa\LL(H)$, respectively). We mention here that there is another, more common meaning of rigidity in the coarse context (see \cite{Roe}).

\subsection{The subgroup exponential hyperballean $\mathcal L(G)$}\label{sub:sub_exp}



First of all, we provide some basic, although very important, examples of $\mathcal L(G)$.  For example, if $G$ is a finite group, then both $\BB_G$ and $\exp\BB_G$ are bounded. In particular, $\mathcal L(G)$ is bounded as well.

\begin{proposition}\label{last:propo} Let $G$ be one of the groups $\Z$ and $\Z_{p^\infty}$ for some prime $p$. Then
\begin{enumerate}[(a)]
\item all balls in $\mathcal L(G)$ are finite;
\item $\mathcal L(G)$ has two connected components, of which one is a singleton $($namely, $\{\{0\}\}$, when $G = \Z$, otherwise $\{G\})$;
\item $\mathcal L(G)$  is thin and thus $\asdim\mathcal L(G)=0$ since it is cellular.
\end{enumerate}
\end{proposition}

\begin{proof} Items (a) and (b) are trivial.

(c) {\bf Case $G=\Z$.} To show that $G = \mathcal L(\Z)$ is thin take an arbitrary finite subset $F$ of $\Z$ and choose $m$  so that $F\subseteq [-m,m]\cap \Z$.
%
%
Pick $n>3m$. We claim that $B_{\mathcal L(\Z)}(n\Z,F)=\{n\Z\}$. We carry out the proof for $F= [-m,m]\cap \Z$, obviously, this implies the general case.

Consider the quotient map $q: \Z\to \Z(n):=\Z/n\Z$ and notice that the subset $q(F)$ of $\Z(n)$ contains no
non-trivial subgroups, by the assumption $3m< n$.  Pick $H\in B(n\Z,F)$, then $q(H) \subseteq q(F)$, so $q(H) = \{0\}$ in $\Z/n\Z$, hence  $H\leq  n\Z$. Thus, $H = l\Z$ for some multiple $l$ of $n$. Since $n\Z\in B(H,F)$, with $l\geq n \geq 3m$, the previous
argument implies $n\Z \leq H$. Therefore, $H = n\Z$.

{\bf Case $G=\Z_{p^{\infty}}$.}  We consider now the group $G = \Z_{p^\infty}$, where $p$ is a prime. 
Denote by $H_{n}$ the subgroup of $\mathbb{Z}_{p^{\infty}}$ of order $p^{n}$, take an arbitrary finite subset $F$ of $\mathbb{Z}_{p^{\infty}}$ and choose $m$  so that $F\subseteq  H_{m}$. Then $B(H_{n}, F)= \{ H_{n}\}$  for each $n> m$.
\end{proof}

\begin{remark}\label{q:finite_balls} Let $G$ be an arbitrary group. According to Fact \ref{fact:finite_balls}, all balls in $\mathcal L(G)$ centered at $\{e_G\}$ are finite. Nevertheless, this is not true for all balls of $\mathcal L(G)$. One can find examples of abelian groups $G$ such that some balls in $\mathcal L(G)$ centred at $G$ are infinite.  For example, let $G=\Pi_{n\in\N}G_n$, where $G_n\simeq\Z/2\Z$, for every $n\in\N$. For every $n\in\N$, denote by $a_n$ the element of $G$ such that $p_n(a_n)=1$ and, for every $i\neq n$, $p_i(a_n)=0$. Then, for every $n\in\N$, $\langle\{a_i: i\in\N\setminus\{1,n\}\}\cup\{a_n+a_1\}\rangle\in B_{\mathcal L(G)}(G,\langle a_1\rangle)$ and thus this ball contains infinitely many elements.
\end{remark}

\begin{corollary}\label{coro:Z_Zp_asy}
$\mathcal L(\Z)$ and $\mathcal L(\Z_{p^\infty})$ are asymorphic, for every prime $p$.
\end{corollary}

\begin{proof} By Proposition \ref{last:propo}(b), both $\mathcal L(\Z)$ and $\mathcal L(\Z_{p^\infty})$ have two connected components, namely,
$$
\QQ_{\mathcal L(\Z)}(\Z),  \ \QQ_{\mathcal L(\Z)}(\{0\})=\{0\}, \ \QQ_{\mathcal L(\Z_{p^\infty})}(\Z_{p^{\infty}})=\{\Z_{p^{\infty}}\}, \ \mbox{  and }\ \QQ_{\mathcal L(\Z_{p^\infty})}(\{0\}).
$$
 Moreover, $\lvert\QQ_{L(\Z)}(\Z)\rvert=\lvert\QQ_{\mathcal L(\Z_{p^\infty})}(\{0\})\rvert=\omega$. Since $\mathcal L(\Z)$ and $\mathcal L(\Z_{p^{\infty}})$ are thin, in particular, also $\QQ_{\mathcal L(\Z)}(\Z)$ and $\QQ_{\mathcal L(\Z_{p^\infty})}(\{0\})$ are thin. Hence, Proposition \ref{prop:charact_thin} implies that $\QQ_{\mathcal L(\Z)}(\Z)$ and $\QQ_{\mathcal L(\Z_{p^\infty})}(\{0\})$ coincide with the ideal balleans associated to the ideals of all their bounded subsets, i.e., finite subsets, namely
\begin{equation}\label{eq:thin}
\QQ_{\mathcal L(\Z)}(\Z)=\BB_\II\quad\mbox{and}\quad \QQ_{\mathcal L(\Z_{p^\infty})}(\{0\})=\BB_{\mathcal J},
\end{equation}
where $\II=[\QQ_{\mathcal L(\Z)}(\Z)]^{<\infty}$ and $\mathcal J=[\QQ_{\mathcal L(\Z_{p^\infty})}(\{0\})]^{<\infty}$.

Fix a bijecton $\varphi\colon\mathcal L(\Z)\to\mathcal L(\Z_{p^\infty})$ such that $\varphi(\{0\})=\Z_{p^\infty}$. We claim that $\varphi$ is an asymorphism.
We can apply Remark \ref{RemIso} and the claim follows once we prove that both $\varphi|_{\QQ_{\mathcal L(\Z)}(\{0\})}$ and $\varphi|_{\QQ_{\mathcal L(\Z)}(\Z)}$ are asymorphisms. While the first restriction is trivially an asymorphism, Fact \ref{fact:asy_ideal} and \eqref{eq:thin} imply that also the second one is an asymorphism.
\end{proof}

In contrast to $\mathcal{L}(\mathbb{Z})$, for $n>1$  $\mathcal{L}(\mathbb{Z}^{n})$ is not weakly thin, and, in particular, it is not thin. To see that $\mathcal{L}(\mathbb{Z}^{n})$, $n>1$, has a non-thin connected component, we put
$F= \{(1,0,\ldots , 0), (0,\ldots , 0)\}$ and note that $2\mathbb{Z}\times  S\in B(\mathbb{Z}\times  S, F)$ for each subgroup $S$ of $\mathbb{Z} ^{n-1}$.

\begin{question}\label{question:LZn_cellular}
Is $\mathcal L(\Z^n)$ cellular for every $n\in\N$?
\end{question}

For every $n\in\N$, denote by $\mathcal R(\Z^n)$ the subballean of $\mathcal L(\Z^n)$ whose support is the family $R(\Z^n)$ of rectangular subgroups of $\Z^n$, i.e., $R(\Z^n)=\{k_1\Z\times\cdots\times k_n\Z : k_1,\dots,k_n\in\Z\}$.
Then, for every $n\in\N$, $\mathcal R(\Z^n)$ is cellular. In fact, it is trivial that $\mathcal R(\Z^n)\aaa\Pi_{i=1}^n\mathcal L(\Z)$ and products of cellular balleans are cellular.

For every {\em locally finite} group $G$ (i.e., every finitely generated subgroup of $G$ is finite), the ballean $\mathcal{B}_{G}$ is cellular, equivalently, $\asdim\mathcal{B}_{G}=0$, so $\exp\mathcal{B}_{G}$ and $\mathcal{L}(G)$ are cellular (Proposition \ref{prop:exp_cellular}).

\begin{question}
\label{q:LG_cellular}  Is the ballean $\mathcal{L}(G)$ cellular for an arbitrary group $G$?
\end{question}

\begin{theorem}
Let $n\in\N$. Then $\mathcal L(\Z^2)\aaa\mathcal L(\Z^n)$ if and only if $n=2$.
\end{theorem}

\begin{proof} We have already proved that $\mathcal L(\Z)\not\aaa\mathcal L(\Z^2)$.
	
Now suppose that $n\ge 3$. Fix, by contradiction, an asymorphism $\varphi\colon\mathcal L(\Z^2)\to\mathcal L(\Z^n)$. As recalled in Remark \ref{RemIso}, $\varphi$ induces asymorphisms between the connected components of those two balleans. Because of Remark \ref{rem:Q_Zn},  one of those restrictions is an asymorphism between $\QQ_{\mathcal L(\Z)}(\Z)$ and $\QQ_{\mathcal L(\Z^2)}(\Z^2)$. However, this is an absurd, since the first ballean is thin, while the second one has not that property.
\end{proof}

\begin{question}
Is it true that $\mathcal L(\Z^n)\aaa\mathcal L(\Z^m)$ if and only if $n=m$?
\end{question}


\subsection{The subgroup logarithmic hyperballean $\LL(G)$ and its asymptotic dimension}\label{sec:logL}

\begin{proposition}\label{prop:Zp_N}
For every prime $p$, $\LL(\Z_{p^\infty})$ is asymorphic to the coproduct of $\N$ and a singleton.
\end{proposition}

\begin{proof} It is easy to check that the subspace $S$ of all finite subgroups of $\mathbb{Z}_{p^\infty}$ is isometric to $(\log p) \mathbb{N}$  with the metric $d(x,y)=\lvert x-y\rvert$, for every $x,y\in(\log p)\N$.
\end{proof}

\begin{corollary}\label{coro:asdim_L_N}
 For every pair of prime numbers $p,q$, $\asdim\LL(\Z_{p^{\infty}})\aaa\asdim\LL(\Z_{q^{\infty}})$. Moreover,
 $$
 \asdim\LL(\Z_{p^\infty})=1.
 $$
\end{corollary}

\begin{theorem}
\label{4.12}  $\asdim\LL(\mathbb{Z})=\infty$.  \end{theorem}

\begin{proof} For distinct primes $p_{1},\dots, p_{n}$, we put $S_{n} = \{ p_{1}^{m_{1}} \cdots  p_{n}^{m_{n}} \mathbb{Z}: m_{i}\in\mathbb{N},\,\forall i=1,\dots, n \},$ and let $\imath(p_{1}^{m_1}\cdots p_{n}^{m_n}\Z)=(m_1,\dots,m_n)\in\N^n$. In the sequel
we denote  the $n$-tuples 
$$
(m_1,\dots,m_n), (m_1^\prime,\dots,m_n^\prime) \in \N^n
$$
by $\bar m, \bar m'$,  and the $n$-tuple  $(\max \{m_1,m_1^\prime\}, \dots,\max\{m_n, m_n^\prime\} )$ by $\max\{ \bar m,  \bar m' \}$.

Equip $\mathbb{N}^{n}$ with the taxi driver metric $d_T$, defined by
$$
d_T(\bar m,\bar m')=\Sigma_i\lvert m_i-m_i^\prime\rvert,
$$
for every pair $(\bar m,\bar m')\in\N^n\times \N^n$.

We prove below that $\imath\colon S_{n}\to \mathbb{N}^{n}$ is an  asymorphism.
%
Since $\asdim\mathbb{N}^{n}=n$, this asymorphism will provide subballeans of arbitrary finite asymptotic dimension of $\LL(\mathbb{Z})$, hence $\asdim\LL(\mathbb{Z})=\infty$.

Consider now a second metric on $\N^n$, namely the logarithmic metric $d_{\log}$, induced from $S_n$ through the bijection $\imath$. More precisely,
$$
d_{\log}(\bar m, \bar m')=d_{\LL(\Z)}(\imath^{-1}(\bar m),\imath^{-1}(\bar m')),
$$
for any pair $\bar m,\bar m'\in\N^n$.  We can assume without loss of generality that $p_1,\dots,p_n$ are greater or equal than the base of the logarithm. We claim that those two metrics induce the same ballean structures on $\N^n$.

First of all, we want to prove that
\begin{equation}
\label{eq:dlog}
d_{\log}(\bar m,\bar m')=\begin{cases}
\begin{aligned} &\sum_{i=1}^n\log(p_i)\max\{m_i-m_i^\prime,0\} &\text{if
$p_1^{m_1}\cdots p_n^{m_n}\ge p_1^{m_1^\prime}\cdots p_n^{m_n^\prime}$,
}
\\
&\sum_{i=1}^n\log(p_i)\max\{m_i^\prime-m_i,0\} &\text{otherwise.}\end{aligned}
\end{cases}
\end{equation}

For $\bar m\in\N^n$, let $A_{\bar m}=p_1^{m_1}\cdots p_n^{m_n}\Z = \imath^{-1}(\bar m)$. Then, for $\bar m ,\bar m' \in\N^n$,
$A_{\bar m}\cap A_{\bar m'} = A_{\max \{  \bar m, \bar m' \}}$, hence
\begin{equation}\label{eq:max}
\max\{\lvert A_{\bar m}:(A_{\bar m }\cap A_{\bar m '})\rvert, \lvert A_{\bar m '}:(A_{\bar m}\cap A_{\bar m '})\rvert\}\! =\!
\lvert A_{\bar s}:A_{\max\{\bar m,\bar m '\}}\rvert=\prod_{i=1}^np_i^{\max\{m_i,m_i^\prime\}-s_i}\!=\!\prod_{i=1}^np_i^{\max\{m_i-s_i,m_i^\prime-s_i\}},
\end{equation}
%
where
\begin{equation}
\label{eq:s}
\bar s=\begin{cases}
\begin{aligned}&\bar m '&\text{if 
$p_1^{m_1}\cdots p_n^{m_n}\ge p_1^{m_1^\prime}\cdots p_n^{m_n^\prime}$,
}\\
&\bar m &\text{otherwise}.
\end{aligned}
\end{cases}
\end{equation}
Hence, \eqref{eq:dlog} can be obtained by combining \eqref{eq:max} and \eqref{eq:s}.

We are left with the proof of $\BB_{d_T}=\BB_{d_{\log}}$. Fix $R\ge 0$ and consider a pair $\bar m, \bar m^\prime\in\N^n$ with $d_T(\bar m, \bar m')\le R.$ Let $K=\max\{\log(p_i): i=1,\dots,n\}.$ By our assumption, $K\ge 1$. Then
$$
d_{\log}(\bar m, \bar m')\leq\sum_{i=1}^n\log(p_i)\lvert m_i-m_i^\prime\rvert\leq RK,
$$
witnessing  that $\BB_{d_{T}}\prec\BB_{d_{\log}}$.

Conversely, let $S\ge 0$ and fix a pair $\bar m, \bar m^\prime\in\N^n$ with $d_{\log}(\bar m, \bar m')\leq S$.
We can assume without loss of generality that that $\Pi_{i\in I}p_i^{m_i}\ge\Pi_{i\in I}p_i^{m_i^\prime}$, where $I=\{1,\dots,n\}$.
Hence,  the first case occurs in \eqref{eq:dlog}.

Split $I=I_+\cup I_-$, with $I_+=\{i\in I: m_i\geq m_i^\prime\}$ and $I_-=\{i\in I: m_i<m_i^\prime\}$. Clearly,  $I_+\ne \emptyset$ due to our assumption (if $I_-= \emptyset$ we set $\prod_{i\in I_-}p_i^{m_i}=1$ below). Hence


$$
\bigg(\prod_{i\in I_+}p_i^{m_i}\bigg)\bigg(\prod_{i\in I_-}p_i^{m_i}\bigg)\ge\bigg(\prod_{i\in I_+}p_i^{m_i^\prime}\bigg)\bigg(\prod_{i\in I_-}p_i^{m_i^\prime}\bigg),
$$
equivalently,
\begin{equation}\label{eq:prod}
\prod_{i\in I_+}p_i^{m_i-m_i^\prime}\ge\prod_{i\in I_-}p_i^{m_i^\prime-m_i}.
\end{equation}
In particular, \eqref{eq:prod} implies $p_j^{m_j^\prime-m_j}\leq\Pi_{i\in I_+}p_i^{m_i-m_i^\prime}$, for every $j\in I_-$, so
\begin{equation}\label{eq:I-}
\lvert m_j-m_j^\prime\rvert=m_j^\prime-m_j=\log_{p_j}(p_j^{m_j^\prime-m_j})\leq\sum_{i\in I_+}(m_i-m_i^\prime)\log_{p_j} p_i\leq \sum_{i\in I}(\log p_i)\max\{m_i-m_i^\prime,0\} = d_{\log}(\bar m, \bar m')\leq S,
\end{equation}
since $\log_{p_j} \leq \log p_i$ as  $p_j$ is greater or equal than the base of the logarithm.
Since, for every $k\in I_+$,
\begin{equation}\label{eq:I-*}
\lvert m_k-m_k^\prime\rvert=m_k-m_k^\prime\leq (m_k-m_k^\prime) \log p_k \leq\sum_{i\in I} (\log p_i)\max\{m_i-m_i^\prime,0\} = d_{\log}(\bar m, \bar m')\leq S,
\end{equation}
the inequalities \eqref{eq:I-} and \eqref{eq:I-*} imply that $d_T(\bar m, \bar m')
\leq nS$.

Hence, $\BB_{d_{\log}}\prec\BB_{d_{T}}$. This proves the equality $\BB_{d_T}=\BB_{d_{\log}}$.
\end{proof}

In particular, Corollary \ref{coro:asdim_L_N} and Theorem \ref{4.12} imply that $\LL(\Z_{p^{\infty}})$ is not even coarsely equivalent to $\LL(\Z)$. Note the difference with Corollary \ref{coro:Z_Zp_asy}.

As we have already noticed, for $n>1$, $\LL(\mathbb{Z}^{n})$  and  $\LL(\mathbb{Z})$   are not asymorphic because $\LL(\mathbb{Z})$   has two connected components but $\LL(\mathbb{Z}^{n})$   infinitely (countably) many. It will be nice to answer the following less obvious question:

\begin{question}\label{[Question 7.]} Are $\LL(\mathbb{Z}^{n})$  and $\LL(\mathbb{Z}^{m})$ asymorphic for all distinct $n, m>1$? \end{question}

In order to characterise the abelian groups $G$ with $\asdim  \LL(G)< \infty$ we need to rule out the groups that are not finitely layered. For a group $G$ and $n\in \N$ let
$$
X _{n}=\{x\in G: o(x)=n\}
$$
(so that  $G[n] = \bigcup_{d|n} X_d$). Note that if $G$ is abelian, then $G[n]$ is a subgroup of $G$ (unlike $X_n$). Call $G$ {\em layerly finite}, if the set $X _{n}$ is finite for every  $n$ (or equivalently, when $G[n] $ is finite for each $n$).

\begin{theorem}\label{4.13} Let $G$  be an abelian group, and $p$ be a prime number. If the subgroup $G[p]$ is infinite then $\asdim\LL(G)=\infty$.
\end{theorem}

\begin{proof} We take a subgroup $H=\bigoplus_{\omega} H_{n}$  of  $G$ which is a direct sum  of $\omega$ copies $H_n$ of $\mathbb{Z}_{p}$,  denote by $\mathcal S$  the set of all subgroups of $H$ of the from $H_F:= \bigoplus_{n\in F} H_{n}$, where $F$  is a finite subset of $\omega,\bigoplus_{n\in\emptyset} H_{n}=\{0\}$.  Then  the correspondence $H_F \mapsto F$ defines an asymorphism between
$\mathcal  S$ and the Hamming space $\mathbb{H}_{\omega}$.  According to Proposition \ref{dimHamm}, $\asdim \ \mathbb{H}_{\omega}=\infty$. Therefore, $\asdim \  \mathcal S=\infty$. This yields $\asdim\LL(G)=\infty$.
\end{proof}

\begin{corollary}\label{New:Corollary}
Let $G$ be an abelian group with $\asdim\LL(G)< \infty$. Then $G$ is torsion and  layerly finite.
\end{corollary}

\begin{proof} By $\asdim\LL(G) <\infty$ and by Theorem \ref{4.12},  $G$ is a torsion group. By Theorem \ref{4.13},  $G$ is layerly finite.
\end{proof}

We can characterise now the abelian groups $G$ such that $\asdim\LL(G)=0$ as the reduced torsion finitely layered
abelian groups. For a  prime $p$ we denote by $S_p$ the {\em Sylow $p$-subgroup} of $G$, i.e.,  is the maximal $p$-subgroup of $G$.

\begin{theorem}\label{[4.14]}  For an abelian group $G$, $\asdim \LL(G)=0$ if and only if $G$ is a torsion group and for every prime $p$ the Sylow $p$-subgroup $S_{p}$ of $G$ is finite.
\end{theorem}

\begin{proof} By Theorem \ref{4.12} and Corollary \ref{New:Corollary}, $G$ is a torsion layerly finite group. If some $S_{p}$ is infinite then $S_{p}$ has a subgroup isomorphic to $\mathbb{Z} _{p^{\infty}}$ but $\asdim\LL( \mathbb{Z} _{p^{\infty}})=1$.


We  assume that each $S_{p}$ is finite and show that $\LL(G)$ is cellular. Let $G=\bigoplus  _{p\in\pi(G)}  S_{p}$. We  take an arbitrary $n\in \mathbb{N}$ and put $G_n=\bigoplus\{S_{p} : p\in \pi(G),\,\log p>n\}$. If $A,B\in  L(G)$ and $d _{\LL(G)} (A, B)\leq  n$ then
$A\cap G_n = B\cap G_n$. It follows that $B^{\Box}(A, n)\subseteq  B(A, m)$, where $m=\sum\{ \log |S_{p}|: p\in \pi(G), \log \ p\leq n\} $.
\end{proof}


Now we show that for every $n\in \N$ one can easily build a (divisible) abelian group $G$ with
$$
\asdim\LL(G) = n.
$$

\begin{example} (a) For distinct primes $p_{1},\ldots   , p_{n}$ consider the group $G= \mathbb{Z}_{p_{1}^{\infty}} \oplus  \ldots \oplus\mathbb{Z}_{p_{n}^{\infty}}$.
Then
$$
\LL(G)\aaa\prod_{i=1}^n\LL(\Z_{p_i^{\infty}})=\coprod_{i=0}^n\bigg(\coprod_{j=1}^{n\choose i}\N^{n-i}\bigg),
$$
so in particular $\asdim \LL(G) = n$. For a proof one has to use the fact that the  lattice $L(G)$ is isomorphic to the direct
product of the lattices  $L(\mathbb{Z}_{p_{1}^{\infty}}) \times  \ldots \times L(\mathbb{Z}_{p_{n}^{\infty}})$ since every subgroup $H$ of $G$ has the form
$$
H = \bigoplus_{i=1}^n H_{p_i}, \mbox{ where $H_{p_i}$ is a subgroup of $\mathbb{Z}_{p_{i}^{\infty}}$}.
$$

(b) More generally, for a set $\pi$ of primes let $G_\pi=\bigoplus_{p\in \pi} \mathbb{Z}_{p^{\infty}}$. Then
$\asdim G_\pi =|\pi|$. Indeed, for finite $\pi$ this follows from (a). Otherwise, consider subsets $\pi_n\subseteq \pi$
with $|\pi_n |=n$ and apply again (a) to the subgroup   $G_{\pi_n}$ to deduce $\asdim G_{\pi_n} =n$ and conclude
$\asdim G_{\pi} = \infty$.
\end{example}

\begin{remark}\label{rem:4.15}
Let $G$ be an abelian group, $n\in \mathbb{N}$. If $\asdim\LL(G) =n$ then $G$ is torsion and there exist distinct primes $p_{1},\dots,p_{m}$, $m\leq n$, a layerly finite subgroup $G_1$ of $G$ which is a direct sum of cyclic subgroups, such that
\begin{equation}\label{last:eq}
G\simeq \mathbb{Z}_{p_{1}^{\infty}} \oplus  \cdots \oplus\mathbb{Z}_{p_{m}^{\infty}} \oplus G_{1}.
\end{equation}

Indeed, by Corollary \ref{New:Corollary}, $G$ is torsion and finitely layered. Hence,  its maximal divisible subgroup $d(G) =  \mathbb{Z}_{p_{1}^{\infty}} \oplus  \cdots \oplus\mathbb{Z}_{p_{m}^{\infty}}$ has $r(G) = m\leq n$. So $G$ splits as in (\ref{last:eq}). Furthermore,
letting $G_2 = \mathbb{Z}_{p_{1}^{\infty}} \oplus  \cdots \oplus\mathbb{Z}_{p_{m}^{\infty}}$, one may have $\pi(G_{1})\cap \pi(G_{2})\ne\emptyset $,
but it is possible to split $G_1 = G_1^* \oplus F$, where $F$ is a finite group, such that, with $G_2^* := G_2 \oplus F$, one has
$$
G = G_1^* \oplus G_2^* \ ,  \  \  G_2^* = \mathbb{Z}_{p_{1}^{\infty}} \oplus  \cdots \oplus\mathbb{Z}_{p_{m}^{\infty}} \oplus F \  \  \mbox{ and } \  \  \  \pi(G_{1}^*)\cap \pi(G_{2}^*)=\emptyset.
$$
\end{remark}


More precise results depend on the following:

\begin{problem} Compute $\asdim\LL(\mathbb{Z}_{p^{\infty}}  \oplus \mathbb{Z}_{p^{\infty}} )$.
\end{problem}

Note that $\LL(\mathbb{Z}_{p^{\infty}}  \oplus \mathbb{Z}_{p^{\infty}})$ contains, as a subspace, the family $\mathcal S\subseteq\QQ_{L(\Z_{p^{\infty}}\oplus\Z_{p^{\infty}})}(\{0\})$ of all proper subgroups of the form $H = H_1\oplus H_2$, where $H_i$ is a proper subgroup of $\mathbb{Z}_{p^{\infty}}$ for $i = 1,2$. Since $\mathcal S\aaa\N^2$,
\begin{equation}\label{eq:asdim_S}\asdim\LL(\mathbb{Z}_{p^{\infty}}\oplus \mathbb{Z}_{p^{\infty}})\ge\asdim\mathcal S=\asdim\N^2=2.
\end{equation}

We do not state in Remark \ref{rem:4.15} that the converse implication is true. More precisely, if $G$ is as in Remark \ref{rem:4.15}
with primes  $p_{1},\dots,p_{m}$ not necessarily distinct,  we cannot claim that $\asdim\LL(G) =n$ is finite with $ n \geq m$.  In case $\asdim \LL(\mathbb{Z}_{p^{\infty}}\oplus \mathbb{Z}_{p^{\infty}})=\infty$ occurs for all primes $p$ (see Remark \ref{last:remark}),
we can claim that $\asdim\LL(G) =n$ entails that  the primes $p_{1},\dots,p_{m}$ are pairwise distinct (and so, $m = n$).

\begin{proposition}\label{voluntary}
Let $G$ be a $p$-group. Consider the subballean of $\LL(G)$ whose support is the family $C(G)$ of all cyclic subgroups of $G$. Then $\asdim C(G)\leq 1$. Moreover, $\asdim C(G)=0$ if and only if $G$ has finite exponent.
\end{proposition}
\begin{proof}
We claim that $C(G)$ is asymorphic to a tree, which implies that $\asdim C(G)\leq 1$ (see \cite[Proosition 9.8]{Roe}). Define a graph $T$ having $C(G)$ as set of vertices and, for $X,Y\in C(G)$, the pair $\{X,Y\}$ is an edge if and only if $X\leq Y$ and $\lvert Y:X\rvert=p$, or $Y\leq X$ and $\lvert X:Y\rvert=p$. Then $T$ is trivially asymorphic to $C(G)$. Obviously,
$(T,\leq)$ is also a partially ordered set, where the order is defined by the inclusion of subgroups. We want to show that $T$ is actually a tree. Consider $X,Y,Z\in C(G)$ such that $Y,Z\leq X$. Let $X=\langle x\rangle$. Since $Y,Z\in C(G)$, $Y=\langle x^{p^y}\rangle$, and $Z=\langle x^{p^z}\rangle$, for some $y,z\in\N$. If $z\leq y$, then $Y\leq Z$ since $x^{p^y}=(x^{p^z})^{p^{y-z}}$. Similarly, if $y\leq z$, then $Z\leq Y$. Since the set $T_X$ of vertices below this fixed vertex $X$ is finite, hence it is well-ordered. This shows that the partially ordered set $(T,\leq)$ is a tree with root the trivial subgroup of $G$ and height equal to the (logarithm of the) exponent of $G$, hence at most $\omega$.

Finally, $\asdim C(G)=0$ if and only if $T$ is bounded and this is equivalent to $G$ having finite exponent.
\end{proof}

For $G=\Z_{p^{\infty}}\oplus\Z_{p^{\infty}}$ we proved $\asdim C(G)=1$ in Proposition \ref{voluntary}. However, $\asdim\QQ_{\LL(G)}(\{0\})\ge 2$, by \eqref{eq:asdim_S}.

Let us conclude our discussion about Remark \ref{rem:4.15} with a final remark
towards an answer to the question whether $\asdim \LL(\mathbb{Z}_{p^{\infty}}  \bigoplus \mathbb{Z}_{p^{\infty}}  )=\infty$.

\begin{remark}\label{last:remark}
Unlike the group  $\mathbb{Z}_{p^{\infty}}  \oplus \mathbb{Z}_{q^{\infty}}$, with primes $p\ne q$, the group $\mathbb{Z}_{p^{\infty}}  \oplus \mathbb{Z}_{p^{\infty}}$ has $\mathfrak c$ many subgroups, actually $\mathfrak c$ many divisible subgroups isomorphic to $\mathbb{Z}_{p^{\infty}}$.
%
%

 One can easily see that $G = \mathbb{Z}_{p^{\infty}}\oplus \mathbb{Z}_{p^{\infty}}$ has three types of subgroups:
\begin{enumerate}[(i)]
\item finite subgroups;
\item infinite proper divisible subgroups, they are all isomorphic to $\mathbb{Z}_{p^{\infty}}$;
\item infinite proper non-divisible subgroups, they are all isomorphic to $\Z(p^n) \oplus \mathbb{Z}_{p^{\infty}}$.
\end{enumerate}

 There are countably many finite subgroups and $\mathfrak c$ many subgroups of each type (ii) and (iii). \\
 We conjecture  that $\asdim \LL(\mathbb{Z}_{p^{\infty}}  \oplus \mathbb{Z}_{p^{\infty}} ) =\infty$.
\end{remark}

\begin{example} \label{ex:Tarskii} A {\em Tarskii monster of exponent $p$}, where $p$ is a prime, is an infinite countable group whose proper subgroups are cyclic and have order $p$. Olshanskii \cite{Ols} built Tarskii monsters for every prime $p > 10^{75}$.

Let $G$ be a Tarskii monster of exponent $p$, where $p$ is a suitable prime. Since every proper subgroup is finite, then $L(G)=\{G\}\sqcup\QQ_{L(G)}(\{e\})$, where $L(G)$ can be endowed both with the subgroup exponential hyperballean structure and with the subgroup logarithmic hyperballean structure.

\begin{enumerate}[(a)]
\item First of all, we focus on the subgroup exponential ballean $\mathcal L(G)$. Fact \ref{fact:finite_balls} implies that every ball centered in a proper subgroup of $G$ is finite. We are not aware whether $\mathcal L(G)$ is thin.
Since $\mathcal L(G)=\{G\}\sqcup \QQ_{L(G)}(\{e\})$, if $\mathcal L(G)$ is thin, then $\QQ_{L(G)}(\{e\})=\BB_{\II}$, where $\II=[\QQ_{L(G)}(\{e\})]^{<\omega}$.

\item We now consider $\LL(G)$. The definition of $G$ implies that the ball centred at the identity of radius $\log p$ contains all proper subgroups of the group, which are infinitely many. Hence, $\LL(G)=\{G\}\sqcup V$, where $V$ is the family of all proper subgroups of $G$, and $V$ is bounded. In particular, $\LL(G)$ is thin and $0$-dimensional.
\end{enumerate}
\end{example}

\begin{remark}\label{rem:inclusions}
For every group $G$, there is a natural map $i\colon G\to L(G)$ that sends every element $g\in G$ in the subgroup $\langle g\rangle$.  (One may consider also the co-restriction $G\to C(G)$ of $i$, where $C(G)$ is the family of all cyclic subgroups of $G$.)
The cardinalities of its fibres have a uniform bound (i.e., $i$ has uniformly bounded fibres) if and only if there is an upper bound for the size of
of all finite cyclic subgroups of $G$ (e.g., the groups of finite  exponent as well as torsion-free groups have this property).
 Hence, one might think that $i$ could be a coarse embedding if $L(G)$ is endowed with a subgroup hyperballean structure. For example, if $G$ is finite, then $i$ is trivially a coarse equivalence. However, if $G$ is infinite this may fail even in simple cases. For $G=\Z$ the map $i\colon G\to L(G)$ is surjective, with $\asdim G=1$, yet $\asdim\mathcal L(G)=0$ and $\asdim\LL(G)=\infty$; hence $i$ is not a coarse equivalence in both cases.

If $G$ is infinite and has finite exponent $n$, then $i\colon G\to\LL(G)$ is not proper. In fact, every cyclic subgroup belongs to the ball in $\LL(G)$ centred in $\{e\}$ with radius $\log n$ (see also Example \ref{ex:Tarskii}) and those subgroups are infinitely many. Hence, $i^{-1}(B_{\LL}(\{e\},\log n))$ is unbounded in $\BB_G$, i.e., infinite.
\end{remark}

\subsection{$G\mbox{-}\exp\BB_\II$}\label{sub:G-space}

\begin{definition}
Let $X$ be a $G$-space with action $G\times X\longrightarrow X$, $(g, x)\longmapsto  gx$, and let $\mathcal{I}$ be a group ideal on $G$. The ballean $\mathcal{B}(G, X, \mathcal{I})$ is defined as $(X , \mathcal{I}, B)$, where  $B(x, A)= Ax \cup\{x\}$ for all $x\in X$,  $A\in \mathcal{I}$.
\end{definition}

By \cite[Theorem 1]{b2}, every ballean $\mathcal{B}$  with support $X$ is asymorphic to  $\mathcal{B}(G, X, \mathcal{I})$  under appropriate choice of $G$  as a subgroup of the group $S_{X}$ of all permutations of $X$  and a group ideal $\mathcal{I}$.

Note that the finitary ballean $\mathcal{B} _{G}$ on a group $G$ is precisely $\mathcal{B} (G,G, [G]^{<\omega})$ with the action of $G$ on $G$  by left translations.

For $\mathcal{B} = \mathcal{B} (G, X, \mathcal{I})$, we introduce a $G$-{\it hyperballean }  $G\mbox{-}\exp \mathcal{B}$ as $(\exp X, \mathcal{I},G\mbox{-}\exp B)$, where
$$
G\mbox{-}\exp B(Y,A)= \{Y\}\cup \{gY: g\in A\},
$$
for every $Y\in\exp X$ and every $A\in\II$. Since $G$ acts by bijections, if $Y$ and $Z$ are two non-empty subsets of $X$, then
\begin{equation}\label{eq:card}
\lvert Y\rvert=\lvert Z\rvert\text{ if there exists $\{g\}\in\II$ such that $Z\in G\mbox{-}\exp B(Y,\{g\})$}.
\end{equation}

\begin{proposition}
For every ballean $\BB=\BB(G,X,\II)$, $G\mbox{-}\exp \mathcal{B}\prec\exp\mathcal{B}$. Moreover, the following properties are equivalent:
\begin{enumerate}[(a)]
\item $G\mbox{-}\exp\mathcal{B}=\exp\mathcal{B}$;
\item each ball in $\exp\mathcal{B}$ around a singleton consists of sigletons;
\item $\BB$ is discrete.
\end{enumerate}
\end{proposition}
\begin{proof}
Fix a radius $A\in\II$ and assume, without loss of generality, that it satisfies $A=A^{-1}$. Then, for every non-empty subset $Y$ of $X$, if $Z\in G\mbox{-}\exp B(Y,\alpha)$, then $Z=gY$ for some $g\in A$. Thus
$$
Z=gY\subseteq AY\quad\mbox{and}\quad Y=g^{-1}Z\subseteq AZ,
$$
which implies that $Z\in\exp B(Y,A)$.

The implications (b)$\to$(c)$\to$(a) are trivial. 
Suppose now that $G\mbox{-}\exp\mathcal{B}=\exp\mathcal{B}$. Then, for every $\{x\}\subseteq X$ and every $A\in\II$, $\lvert Y\rvert=1$, provided that $Y\in\exp B(\{x\},A)$, because of \eqref{eq:card}.
\end{proof}

\begin{proposition}
\label{prop:dsc_expG}
For an infinite group $G$, $\dsc(G\mbox{-}\exp\BB_G)=\dsc(\exp\BB_G)=2 ^{\lvert G\rvert}$.
\end{proposition}

\begin{proof} Let $T$ be a thin subset of $G$ such that $\lvert T\rvert=\lvert G\rvert$, whose existence is proved in \cite{b1}.
%
%

Since $T$ is a thin subset of $G$, Proposition \ref{prop:charact_thin} implies that $\BB_G|_T$ coincides with the ideal ballean $\BB_\II$, where $\II$ is the ideal of all bounded subsets of $T$ (i.e., all finite subsets of $T$). By \cite[Proposition 4.1]{b5} the equivalence relation $A\sim B$, where $A,B\subseteq T$, if and only if $A\triangle B$ is finite, is precisely the equivalence relation of belonging to the same connected component in $\exp(\BB_G|_T)$. Hence, each connected component of $\exp(\BB_G|_T)$ has cardinality precisely $\lvert T\rvert=\lvert G\rvert$ and thus there are $2^{\lvert G\rvert}$ such connected components. Finally, since $G\mbox{-}\exp\BB_G\prec\exp\BB_G$,
$$
2^{\lvert G\rvert}\ge\dsc(G\mbox{-}\exp\BB_G)\ge\dsc(\exp\BB_G)\ge\dsc(\exp(\BB_G|_T))=2^{\lvert G\rvert}.
$$
\end{proof}
\begin{example} Denote by $S_\omega$ the group of all permutations of $\omega$. Let us take the ballean $\mathcal{B}=\mathcal{B}(S_{\omega},\omega,[S_\omega]^{<\omega})$
and show that $\exp\mathcal{B}$ has only two connected components: the family of all non-empty finite subsets of $\omega$ and the family of all infinite ones. 

For any two non-empty finite subset $X_{1} , X_{2}$ of $\omega$ and each $x\in X_{1}$, $y\in X_{2}$, let $s _{x,y}$ be the transposition
with support $\{x,y\}$, i.e., $s_{x,y}(x)=y$, $s_{x,y}(y)=x$ and $s_{x,y}|_{X\setminus\{x,y\}}=id|_{X\setminus\{x,y\}}$. Then
$X_1 \in \exp B (X_2,F)$ with respect to $F= \{s_{x,y}: x\in X_1, y\in X_2\}$.

We take an arbitrary infinite subset $Y$ of $\omega$, partition $Y$ into infinite subsets $Y=Y _{1} \cup Y_{2}$, and partition $\omega= W_{1} \cup W_{2}$  so that $Y_{1} \subseteq W_{1}$, $Y_{2} \subseteq W_{2}$. Then  we choose two permutations 
$f_{1}, f_{2}$  of $\omega$ so that $f_1(Y_{1})= W_{2}$, $f_2(Y_{2})= W_{1}$  and put  $F=\{id_{\omega}, f_{1}, f_{2}\}$. Then $\omega\in \exp B(Y, F)$.

In contrast to $\exp\mathcal{B}$, the ballean  $G\mbox{-}\exp\mathcal{B}(S_{\omega},  \omega, [S_\omega]^{<\omega})$  has countably
many connected components:

\vskip3pt

\centerline{$\{F\subseteq\omega : \lvert F\rvert=n\}$,  $n\in \omega$, $n> 0$,  $\{Y \subseteq\omega : \lvert\omega\setminus Y\rvert=n\}$,  $n\in \omega$    and $\{Y \subseteq\omega : \lvert Y\rvert=\lvert\omega \setminus Y\rvert=\omega\}$.}
\end{example}

 Since non-trivial cosets of a subgroup are never subgroups, the subballean $G\mbox{-}\exp\BB_G|_{L(G)}$ is trivial and so irrelevant for the purpose of this paper.


\section{Rigidity results}\label{sec:rigidity}
As we have already mentioned (see comments on Remark \ref{rem:appl_iso}), if two groups $G$ and $H$ are isomorphic, then $\mathcal L(G)\aaa\mathcal L(H)$ and $\LL(G)\aaa\LL(H)$. However, the converse is not true in general (for example, $\mathcal L(\Z)\aaa\mathcal L(\Z_{p^{\infty}})\aaa\mathcal L(\Z_{q^{\infty}})$ and $\LL(\Z_{p^{\infty}})\aaa\LL(\Z_{q^{\infty}})$). In this section we want to determine conditions that ensures that the opposite implication holds.

Let us start with some technical results which hold for the subgroup hyperballeans $\mathcal L(G)$ and $\LL(G)$.
\begin{lemma}\label{lemma:asy_ce_conncomp}
Let $X$ be a ballean.
\begin{enumerate}[(a)]
	\item If $X$ is asymorphic to $\mathcal L(\Z)$ or to $\LL(\Z)$, then $X$ has two connected components. Moreover, one connected component is a singleton, while the other one is infinite and unbounded.
	\item If $X$ is coarsely equivalent to $\mathcal L(\Z)$ or to $\LL(\Z)$, then $X$ has two connected components. Moreover, one connected component is bounded, while the other one is unbounded.
\end{enumerate}
\end{lemma}
\begin{proof}
The proof is an application of Fact \ref{fact:ce_bounded_unbounded}, Remark \ref{rem:Q_Zn} and Proposition \ref{last:propo}(b).
\end{proof}

An infinite group is said to be {\em quasi-finite} if every proper subgroup is finite. Example of quasi-finite groups are the Pr\" uffer $p$-groups and the Tarskii monsters (see Example \ref{ex:Tarskii}). Moreover, if an abelian group is quasi-finite, then it is isomorphic to Pr\" uffer $p$-group for some prime $p$.
\begin{proposition}\label{lemma:as_in_mathcalL}
	Let $G$ be a group. Suppose that $\mathcal L(G)$ ($\LL(G)$, equivalently) has precisely two connected components, one of them is a singleton and the other one is infinite. Then $G$ must be infinite. Moreover:
	\begin{enumerate}[(a)]
		\item if $G$ contains an element of infinite order then $G\simeq\mathbb{Z}$;
		\item if $G$ is a torsion group then $G$ is quasi-finite.
	\end{enumerate}
\end{proposition}
\begin{proof}
The first statement is trivial, since, otherwise, $\mathcal L(G)$ and $\LL(G)$ would be bounded.
	
(a) Let $g$ be element of infinite order of $G$. Then $\langle g\rangle\in L(G)$ is infinite, $\langle g\rangle\in\mathcal Q_{L(G)}(G)$ and thus $\QQ_{L(G)}(G)$ is infinite (as it contains the subgroups of the form $\langle g^k\rangle$, where $k\in\N$), while $\QQ_{L(G)}(\{e_G\})=\{e_G\}$. Since each infinite subgroup of $G$ is, in particular, large in $G$, it has finite index and, by Fedorov's theorem \cite{b6}, $G\simeq \mathbb{Z}$.
%

(b) Since $G$ is torsion, for every $g\in G$, $\langle g\rangle$ is a finite subgroup and thus belongs to the connected component $\QQ_{L(G)}(e_G)$. Hence, the connected component of $G$ is a singleton and every proper subgroup is finite.
\end{proof}

\subsection{Rigidity results on the subgroup exponential hyperballean $\mathcal L(G)$}\label{sub:rig_sub_exp}

\begin{corollary}
\label{theo:4.3} If a group $G$ contains an element of infinite order, then $\mathcal{L}(G)\aaa \mathcal{L}(\mathbb{Z})$ if and only if $G \simeq \mathbb{Z}$.
\end{corollary}
\begin{proof}
Lemma \ref{lemma:asy_ce_conncomp}(a) implies that $\mathcal L(G)$ has two connected components, one is infinite and the other one is just a singleton. Hence the conclusion follows from \ref{lemma:as_in_mathcalL}(a).
\end{proof}


\begin{theorem}\label{theo:4.5}
 For an abelian group $G$, $\mathcal{L}(G)\aaa \mathcal{L}(\mathbb{Z})$ if and only if  either $G\simeq \mathbb{Z}$  or  $G\simeq \mathbb{Z}_{p^{\infty}}$, for some $p$  is prime.
\end{theorem}

\begin{proof}
The ``if part'' of the statement is proved in Corollary \ref{coro:Z_Zp_asy}.

Conversely, let us divide the proof in two cases. If $G$ is torsion, then Lemmas \ref{lemma:asy_ce_conncomp}(a) and \ref{lemma:as_in_mathcalL}(b) imply that every proper subgroup of $G$ is finite. Hence, since $G$ is abelian, $G\simeq\Z_{p^\infty}$, for some prime $p$. Otherwise, there exists and element $g\in G$ of infinite order and then the claim follows from Corollary \ref{theo:4.3}.
%
%
\end{proof}
Can we relax the hypothesis of Theorem \ref{theo:4.5}? Namely, we wonder whether the request of $G$ being abelian can be relaxed or not. Let us state it as a question.

\begin{question}\label{Question 4}
Let $G$ be a torsion group such that $\mathcal{L}(G)$ and $\mathcal{L}(\mathbb{Z})$ are asymorphic. Is $G\simeq \mathbb{Z}_{p^\infty}$ for some prime $p$?
\end{question}

An affirmative answer to the question of Example \ref{ex:Tarskii}, along with a proof similar to that of Corollary \ref{coro:Z_Zp_asy}, would show that $\LL(T)\aaa\LL(\Z)$,  for  a Tarskii monster $T$. This would provide a negative answer to Question \ref{Question 4}.

\subsection{Rigidity results on the subgroup logarithmic hyperballean $\LL(G)$}\label{sub:rig_sub_log}

\begin{theorem}\label{theo:4.11}
 Let $G$ be a group and $p$ be a prime.
	\begin{enumerate}[(a)]
		\item $\LL(G)\aaa\LL(\mathbb{Z})$ if and only if $G\simeq \mathbb{Z}$;
		\item $\LL(G)\aaa\LL(\mathbb{Z}_{p^{\infty}})$ if and only if $G\simeq \mathbb{Z}_{q^{\infty}}$ for some prime $q$.
	\end{enumerate}
\end{theorem}
\begin{proof}
	(a) 
	Assume that $\LL(G)$ is asymorphic to $\LL(\mathbb{Z})$. If $G$ has an element of infinite order then $G\simeq\mathbb{Z}$, by Lemma \ref{lemma:asy_ce_conncomp}(a) and Proposition \ref{lemma:as_in_mathcalL}(a). Suppose now, by contradiction, that $G$ is a torsion group. By Proposition \ref{lemma:as_in_mathcalL}(b), $G$ is quasi-finite. We show that $G$ is layerly finite. If $A, B$  are subgroup of order  $n$ then $A\subseteq AB$, $B\subseteq BA$  so $d_{\LL(G)} (A,B)\leq\log n$.  If some $X_{n}$ is infinite then $\LL(G) $  has an infinite ball of radius $\log n$, but each ball in $\LL(\mathbb{Z})$ is finite.
		 By \cite{b7}, $G$ either has a subgroup $H, H \simeq  \mathbb{Z}_{p^{\infty}}$  or $G$ is the subdirect product of finite groups.
	  Since this implies the existence of proper (normal) subgroups of finite index
and $G$ is quasi-finite, the second case is impossible. So we are left with $H \simeq  \mathbb{Z}_{p^{\infty}}$.
Since $H$ is infinite and $G$ is quasi-finite, $H\simeq G$. This contradicts the conjunction of Corollary \ref{coro:asdim_L_N} \& Theorem \ref{4.12}.
	
	(b) Corollary \ref{coro:asdim_L_N} implies that $\LL(\Z_{p^\infty})\aaa\LL(\Z_{q^\infty})$ for every pair of primes $p$ and $q$. Conversely, suppose that $\LL(G)\aaa\LL(\Z_{p^\infty})$. If, by contradiction, $G$ contains an element of infinite order, then $G\simeq\Z$, by Proposition \ref{lemma:as_in_mathcalL}(a). This contradicts $\LL(\Z_{p^\infty})\not \approx \LL(\Z)$ established in Corollary \ref{coro:asdim_L_N} and Theorem \ref{4.12}. Hence $G$ is torsion. Using Proposition \ref{lemma:as_in_mathcalL}(b) as above, we conclude that $G$ is quasi-finite and layerly finite, and consequently,
$H \simeq  \mathbb{Z}_{q^{\infty}}$ for some prime $q$.
%
\end{proof}

Note that in Theorem \ref{theo:4.11} we don't require that the group is abelian.

\subsection{Rigidity results and questions on divisible and finitely generated abelian groups}\label{sub:rig_div}

We pointed out in \S\ref{sub:isoLG} that divisibility of a group is related to some strong property of its hyperballean. So it is natural to ask if we can find some rigidity result in this setting.

\begin{lemma}\label{lemma:no_N_in_Q} For no cardinal $\kappa$, $\LL(\Q^\kappa)$ has a connected component asymorphic to $\N$.
\end{lemma}
\begin{proof}
Let $H$ be an arbitrary subgroup of $\Q^\kappa$ and suppose that $H$ is not divisible since, otherwise, $\QQ_{L(\Q^\kappa)}(H)=\{H\}\not\aaa\N$. Since $H$ is not divisible, there exists $n\in\N$ such that $nH\lneq H$. Note that, this is equivalent to $H\lneq(1/n)H$. Hence, in particular, we can construct a chain of subgroups as follows:
$$
\cdots\lneq n^kH\lneq\cdots\lneq n^2H\lneq nH\lneq H\lneq \frac{1}{n}H\lneq \frac{1}{n^2}H\lneq\cdots\lneq\frac{1}{n^k}H\lneq\cdots.
$$
 Note that this chain is asymorphic to $\Z$, which is not asymorphic to $\N$ and this observation concludes the proof.
\end{proof}

\begin{proposition}\label{prop:tor_free}
Let $D$ and $D^\prime$ be two divisible abelian groups. Then $D$ is torsion-free if and only if $D^\prime$ is torsion-free, provided that $\LL(D)\aaa\LL(D^\prime)$.
\end{proposition}
\begin{proof}
Suppose that $D$ is torsion-free.
Let $D^\prime$ have torsion. Then, by Fact \ref{theo:divis_structure},
$$
D^\prime\simeq\Q^{r_0(D^\prime)}\oplus\Z_{p^{\infty}}\oplus H,
$$
where $p$ is a prime and $H\leq t(D^\prime)$. Define $K=\Q^{r_0(D^\prime)}\oplus\{0\}\oplus H$. Then $\QQ_{\LL(D^\prime)}(K)\aaa\N$. As $D$ is torsion-free, $D\simeq\Q^{r_0(D)}$ and there is no connected component asymorphic to $\N$, by Lemma \ref{lemma:no_N_in_Q}.
\end{proof}

 Moreover, we can prove a stronger version of Remark \ref{rem:appl_iso}(c) and (d).

\begin{corollary}\label{prop:rig_div} Let $G$ be a divisible abelian group.
	\begin{enumerate}[(a)]
		\item Then $\LL(G)\aaa\LL(\Q)$ if and only if $G\simeq\Q$.
		\item Suppose that $\kappa$ is an infinite cardinal. Then, under the assumption of the Generalised Continuum Hypothesis, $\LL(G)\aaa\LL(\Q^\kappa)$ if and only if $G\simeq\Q^\kappa$.
	\end{enumerate}
\end{corollary}

\begin{proof} Proposition \ref{prop:tor_free} implies that $G$ is torsion-free, provided $\LL(G)\aaa\LL(\Q)$ or
$\LL(G)\aaa\LL(\Q^\kappa)$ and thus we can apply Remark \ref{rem:appl_iso}
to prove both claims.
\end{proof}

\begin{question}\label{q:divisible}
Let $G$ be an abelian group and $D$ be a divisible abelian group. Is it true that $G$ is divisible, provided that $\LL(G)\aaa\LL(D)$?
\end{question}

Let $G$ and $D$ be as in Question \ref{q:divisible}. Then
$G\simeq d(G)\oplus H$ for some subgroup $H$ of $G$. By Corollary \ref{coro:iso_divis1}, $\iso(D)\neq\emptyset$.
This, along with Corollary \ref{coro:iso_divis}, implies  $t(H)=\{0\}$.

\begin{question}
Let $G$ be an abelian group such that $t(d(G))=\{0\}$. Is it true that $G\simeq\Q$, provided that either $\mathcal L(G)\aaa\mathcal L(\Q)$ or $\LL(G)\aaa\LL(\Q)$?
\end{question}

\begin{question}
Is it true that $\mathcal L(\Q)\aaa\mathcal L(\Q\oplus\Z_{p^\infty})$?
\end{question}
\begin{question}
Is it true that $\mathcal L(\Q\oplus\Z)\aaa\mathcal L(\Q)$ or $\LL(\Q\oplus\Z)\aaa\LL(\Q)$?
\end{question}

\begin{question}\label{q:fin_generated}
Let $G$ be an abelian group and $H$ be a finitely generated abelian group. Suppose that $\LL(G)\aaa\LL(H)$. Is it true that $G$ is finitely generated?
\end{question}

Note that $\mathcal L(\Z_{p^\infty})\aaa\mathcal L(\Z)$, where $\Z$ is finitely generated and it is not divisible, while $\Z_{p^\infty}$ is not finitely generated, although it is divisible.
This is why we formulate Questions \ref{q:divisible} and \ref{q:fin_generated} only for the subgroup logarithmic hyperballean $\LL(G)$

\subsection{Results on coarsely equivalent subgroup exponential hyperballeans}\label{sub:rig_ce}

\begin{lemma}\label{lemma:ce_groups_to_ce_hyp}
Let $G$ and $H$ be two groups.
\begin{enumerate}[(a)]
\item If there exist two homomorphisms $f\colon G\to H$ and $g\colon H\to G$ such that $f\circ g\sim 1_H$ and $g\circ f\sim 1_G$, then $f\colon\BB_G\to\BB_H$ is a coarse equivalence, with coarse inverse $g\colon\BB_H\to\BB_G$, and $\mathcal L(f)=\exp f|_{L(G)}\colon\mathcal L(G)\to\mathcal L(H)$ is a coarse equivalence, with inverse $\mathcal L(g)\colon\mathcal L(H)\to\mathcal L(G)$.
\item Let $H$ be a finite normal subgroup of $G$. Then the quotient map $q\colon \mathcal L(G)\to \mathcal L(G/H)$ is a coarse equivalence and, moreover, $\mathcal L(q)\colon\mathcal L(G)\to\mathcal L(G/H)$ is a coarse equivalence.
\end{enumerate}
\end{lemma}
\begin{proof}
(a) Note that $f\colon\BB_G\to\BB_H$ is trivially a coarse equivalence. Moreover, it is easy to check that $\exp f\colon\exp\BB_G\to\exp\BB_H$ is a coarse equivalence with inverse $\exp g\colon\exp\BB_H\to\exp\BB_G$ (see also \cite{b5}). Since both $f$ and $g$ are homomorphisms, the restrictions $\mathcal L(f)$ and $\mathcal L(g)$ are well-defined and thus they are coarse equivalences.

(b) Since $q$ is a homomorphism, $q\colon\BB_G\to\BB_{G/H}$ is coarse and also a coarse equivalence, since preimages of finite subsets are finite (see Example \ref{ex:ball}(c)). In particular,
$$
\mathcal L(q)=\exp q|_{\mathcal L(G)}\colon\mathcal L(G)\to\mathcal L(G/H),
$$
which is well-defined, is coarse. Moreover, $g\colon \mathcal L(G/H)\to\mathcal L(G)$ defined by the law $g(K)=q^{-1}(K)$, where $K\leq G/H$, is coarse and a coarse inverse of $\mathcal L(q)$.
\end{proof}

\begin{theorem}\label{theo:4.6}
Let a group $G$ contains an element $g$ of infinite order. Then $\mathcal{L}(G)$ and $\mathcal{L}(\mathbb{Z})$ are coarsely equivalent if and only if $G$ has a finite normal subgroup $H$ such that $G/H\simeq\mathbb{Z}$.
\end{theorem}

\begin{proof}
%
($\to$) Assume that $\mathcal{L}(G)$ and $\mathcal{L}(\mathbb{Z})$  are coarsely equivalent. Lemma \ref{lemma:asy_ce_conncomp}(b) implies that $\mathcal L(G)$ has two connected components: one is unbounded (hence, infinite) and one is bounded. Let us see that the connected component
$C:= \QQ_{\mathcal L(G)}(\{e\})$ of $\{e\}$ is the bounded one.
To prove that $C$ is bounded it is enough to observe that it does not contain the infinite subgroup $\langle g\rangle$ as well as its infinitely many
proper subgroups $\langle g^n\rangle$, where $n\ge 2$. Since this family is certainly unbounded in $\mathcal L(G)$, $C$ must be the
bounded component. Consequently, $C$ is finite being contained into a ball around $\{e\}$ (see Fact \ref{fact:finite_balls}).

%
%
%
%

 Since $C$ contains all  finite order elements $h\in G$,  this will imply that the set $H$ of all the elements of finite order of $G$ is finite.
By Ditsmans lemma \cite{b8}, $H$ is a subgroup. Moreover, since conjugacy doesn't change the order of an element, $H$ is normal in $G$. Then $G/H$ is torsion free.


Since $\mathcal L(G/H)$ is coarsely equivalent to $\mathcal L(G)$ (Lemma \ref{lemma:ce_groups_to_ce_hyp}) and thus to $\mathcal L(\Z)$, in particular, we can reapply the usual argument and prove that every proper subgroup $K$ of $G/H$ is large in $G/H$ and so $\lvert G/H:K\rvert$ is finite. By Federov's theorem, $G/H$ is isomorphic to $\Z$.

($\gets$) On the other hand, if $H$ is finite and $G/H\simeq\mathbb{Z}$ then $G=\langle a\rangle H$, $\langle a\rangle\simeq\mathbb{Z}$   and $\mathcal{L}(\langle a\rangle)$ is large in $\mathcal{L}(G)$, so $\mathcal{L}(G) $ and $\mathcal{L}(\mathbb{Z})$  are coarsely equivalent.
\end{proof}

\begin{lemma}\label{lemma:ex1}\label{lemma:ex2}
Let $G$ be a group.
\begin{enumerate}[(a)]
\item if $H$ is a subgroup of $G$ of finite index, then $G$ has only finitely many subgroups containing $H$;
\item if $\mathcal H$ is a family of subgroups of $G$ stable under finite intersections, and there exists $n\in\N$ such that
$\lvert G:H\rvert\leq n$ for every $H\in\mathcal H$, then $\mathcal H$ is finite.
\end{enumerate}
\end{lemma}

\begin{proof}
(a) Let $H_G$ be the core of $H$ in $G$ (i.e., the biggest normal subgroup of $G$ which is contained in $H$), which has still finite index in $G$. Consider the map $q\colon G\to G/H_G$. Then $q$ induces a bijection between the family of subgroups of $G$ containing $H_G$ and the one of the subgroups of $G/H_G$. Since the latter is finite, we are done.

(b) Assume for contradiction that $\mathcal H$ has infinitely many pairwise distinct members
$\{H_m\}_{m\in\N}$. Using the stability of $\mathcal H$ under finite intersections, we can
replace $H_m$ by the intersection $H_m^*:= H_1\cap\cdots\cap H_m$ in order to obtain a
decreasing chain of members of $\mathcal H$. As $\lvert G:H\rvert\leq n$ for every $H\in\mathcal H$, this
chain stabilizes at some stage $N_{m_0}^*$. Therefore all subgroups $H_m$ contain $N_{m_0}^*$. This contradicts (a).
\end{proof}

\begin{theorem}\label{[4.8]}
For an abelian group $G$, $\mathcal{L}(G)$ and $\mathcal{L}(\mathbb{Z})$ are coarsely equivalent if and only if there exists a finite subgroup $H$ of $G$ such that either $G/H\simeq\Z$ or $G/H\simeq\Z_{p^{\infty}}$, for some prime $p$.
\end{theorem}

\begin{proof} Assume that  $\mathcal{L}(G)$ and $\mathcal{L}(\mathbb{Z})$  are coarsely equivalent. If $G$ has an element of infinite order then we apply Theorem \ref{theo:4.6}. Otherwise, suppose that $G$ is a torsion group.
Since $\mathcal L(G)$ and $\mathcal L(\Z)$ are coarsely equivalent, we deduce from Lemma  \ref{lemma:asy_ce_conncomp}, that $\mathcal L(G)$ has two connected components and one of them is bounded, while the other one is unbounded. Since $G$ is torsion, Fact \ref{fact:finite_balls} implies that $\QQ_{\mathcal L(G)}(\{0\})$ must be unbounded. Hence, the family $\mathcal H$ of all finite index subgroups of $G$ satisfies the hypothesis of Lemma \ref{lemma:ex2}(b) and thus $\mathcal H$ is finite and, in particular, it has a minimum element $K$. Then $G/K$ is finite and $K$ is quasi-finite and thus, since $G$ is abelian, $K\simeq\Z_{p^{\infty}}$, for some prime $p$. Hence the claim follows.
%
%
\end{proof}

We cannot state similar results for the subgroup logarithmic hyperballean, since
 the balls centred at $\{0\}$ can have infinitely many elements (see Example \ref{ex:Tarskii}).

\end{document}